\documentclass[a4paper,12pt]{scrartcl}
\usepackage{mathtools}
\usepackage[T1]{fontenc}
\usepackage[utf8]{inputenc}
\usepackage[a4paper]{geometry}
\usepackage{enumitem}
\usepackage{amsthm,amssymb}
\usepackage[backend=bibtex,maxnames=10,defernumbers=true,firstinits=true,sortcites=true]{biblatex}
\usepackage[unicode,hyperfootnotes=false,pdftex]{hyperref}
\usepackage{hyperxmp}
\usepackage{bookmark}
\usepackage[cal=boondoxo]{mathalfa}
\usepackage{microtype}
\usepackage{xparse}
\usepackage{csquotes}
\usepackage{etoolbox}
\usepackage{interval}

\hypersetup{%
        pdftitle={The non-linear sewing lemma I: weak formulation},
        pdfauthor={Antoine Brault; Antoine Lejay},
        pdfdate={2019-02-24},
	pdfkeywords={rough paths, rough flows, selection theorem},
}

\newcommand{\run}{\mathrm{I}}
\newcommand{\rdeux}{\mathrm{II}}
\newcommand{\rtrois}{\mathrm{III}}
\newcommand{\rquatre}{\mathrm{IV}}

\DeclarePairedDelimiterXPP{\normf}[2]{}{\|}{\|}{_{#1}}{#2}
\DeclarePairedDelimiterXPP{\normsup}[1]{}{\|}{\|}{_{\infty}}{#1}
\DeclarePairedDelimiterXPP{\normp}[1]{}{\|}{\|}{_{p}}{#1}
\DeclarePairedDelimiterXPP{\normpp}[1]{}{\|}{\|}{_{\frac{p}{2}}}{#1}
\DeclarePairedDelimiterXPP{\normhold}[2]{}{\|}{\|}{_{#1}}{#2}

\DeclarePairedDelimiterXPP{\generalnorm}[2]{}{\|}{\|}{_{#1}}{#2}

\DeclarePairedDelimiter{\eucnorm}{|}{|}

\newtheorem{proposition}{Proposition}
\newtheorem{theorem}{Theorem}

\newtheorem{lemma}{Lemma}
\newtheorem{corollary}{Corollary}
\theoremstyle{definition}
\newtheorem{definition}{Definition}
\newtheorem{notation}{Notation}

\newtheorem{hypothesis}{Hypothesis}

\theoremstyle{remark}
\newtheorem{remark}{Remark}
\newtheorem{example}{Example}

\newcommand{\vd}{\,\mathrm{d}}
\newcommand{\dd}{\mathrm{d}}
\newcommand{\RR}{\mathbb{R}}

\newcommand{\TT}{\mathbb{T}}

\newcommand{\intSS}{\mathbb{S}}
\newcommand{\NN}{\mathbb{N}}

\newcommand{\bx}{\mathbf{x}}

\newcommand{\uU}{\mathrm{U}}
\newcommand{\uW}{\mathrm{W}}
\newcommand{\uT}{\mathrm{T}}

\newcommand{\uV}{\mathrm{V}}
\newcommand{\uX}{\mathrm{X}}

\newcommand{\uL}{\mathrm{L}}
\newcommand{\cL}{\mathcal{L}}

\newcommand{\cC}{\mathcal{C}}
\newcommand{\cCb}{\mathcal{C}_{\mathrm{b}}}
\newcommand{\cK}{\mathcal{K}}
\newcommand{\cF}{\mathcal{F}}

\newcommand\given{\nonscript\:\delimsize\vert\nonscript\:\mathopen{}} 
\newcommand\SetSymbol[1][]{\nonscript\:#1\vert\nonscript\:\mathopen{}\allowbreak}
\DeclarePairedDelimiterX\Set[1]\{\}{%
  \renewcommand\given{\SetSymbol[\delimsize]}#1}

\DeclarePairedDelimiterX\Fam[1]\{\}{%
  \renewcommand\given{\SetSymbol[\delimsize]}#1}

\newcommand{\eqdef}{\mathbin{:=}}

\pgfkeys{
  /interval/.is family,%
  /interval,%
    int left fence/.initial={\lbracket},
    int right fence/.initial={\rbracket},
    int/.style={left fence=\llbracket,right fence=\rrbracket}
}

\intervalconfig{soft open fences}

\DeclarePairedDelimiterXPP{\normlip}[1]{}{\|}{\|}{_\mathrm{Lip}}{#1}

\newcommand{\rTT}{\TT^+}

\DeclarePairedDelimiter{\abs}{|}{|}

\NewDocumentCommand \ITVArgument
{ m > { \SplitArgument { 1 } { , } } m }
{ \csuse{#1}#2 }

\newcommand{\itvOpenLeft}{(}
\newcommand{\itvOpenRight}{)}
\newcommand{\itvCloseLeft}{[}
\newcommand{\itvCloseRight}{]}
\DeclarePairedDelimiterX{\itvoo}[1]{\itvOpenLeft}{\itvOpenRight}{\ITVArgument{ITVFormatComma}{#1}}
\DeclarePairedDelimiterX{\itvoc}[1]{\itvOpenLeft}{\itvCloseRight}{\ITVArgument{ITVFormatComma}{#1}}
\DeclarePairedDelimiterX{\itvcc}[1]{\itvCloseLeft}{\itvCloseRight}{\ITVArgument{ITVFormatComma}{#1}}
\DeclarePairedDelimiterX{\itvco}[1]{\itvCloseLeft}{\itvOpenRight}{\ITVArgument{ITVFormatComma}{#1}}

\begin{document}

\title{The non-linear sewing lemma I: weak formulation}
\author{Antoine Brault\thanks{Université Paris Descartes, MAP5 (CNRS UMR 8145), 45 rue desSaints-Pères, 75270 Paris cedex 06, France.
\texttt{antoine.brault@parisdescartes.fr}}{\ }\thanks{Institut de Mathématiques de Toulouse, UMR 5219; Université de Toulouse, UPS IMT, F-31062 Toulouse Cedex 9, France.} \and Antoine Lejay\thanks{Université de Lorraine, CNRS, Inria, IECL, F-54000 Nancy, France, \texttt{antoine.lejay@univ-lorraine.fr}}}
\date{Feburary 24, 2019}
\maketitle

\begin{abstract}
We introduce a new framework to deal with rough differential equations based on
flows and their approximations.  Our main result is to prove that measurable
flows exist under weak conditions, even if  solutions to the corresponding
rough differential equations are not unique.
We show that under additional conditions of the approximation, there exists a unique Lipschitz flow. Then, a perturbation formula is given. Finally, we link our approach to the additive, multiplicative sewing lemmas and the rough Euler scheme.
\end{abstract}


\section{Introduction}
\subsection{Motivations}
The theory of rough paths allows one to define the solution to a differential 
equation of type 
\begin{equation}
    \label{eq:rde:3}
    y_t=a+\int_0^t f(y_s)\vd x_s,
\end{equation}
for a path $x$ which is irregular, say $\alpha$-Hölder continuous. 
Such an equation is then called a \emph{Rough Differential Equation} (RDE)
\cite{friz,lyons02b,friz14a}.
The key point
of this theory is to show that such a solution can be defined provided that~$x$
is extended to a path~$\bx$, called a \emph{rough path}, living in a larger space
that depends on the integer part of $1/\alpha$. When $\alpha>1/2$, no such extension
is needed. This case is referred as the \emph{Young case}, as the integrals 
are constructed in the sense given by L.C. Young \cite{young36a,lyons94a}.
Provided that one considers a rough path, integrals and differential equations are 
natural extensions of ordinary ones. 

The first proof of existence of a solution to \eqref{eq:rde:3}
from T.~Lyons relied on a fixed point \cite{lyons06b,lyons02b,lyons98a}.
It was quickly shown that RDE shares the same properties as ordinary differential 
equations, including the flow property.
In \cite{davie05a}, A.M. Davie gives an alternative proof based on 
an Euler type approximation as well as counter-example to uniqueness. 
More recently, I.~Bailleul gave a direct construction through the flow property
\cite{bailleul12a,bailleul14a,bailleul15a}. 

A \emph{flow} in a metric space  $\uV$
is a family $\Fam{\psi_{t,s}}_{0\leq s\leq t\leq T}$ of maps 
from $\uV$ to $\uV$
such that $\psi_{t,s}\circ\psi_{s,r}=\psi_{t,r}$ for any $0\leq r\leq s\leq t\leq T$. 
When $y_t(s,a)$ is a family of solutions to differential equations
with $y_s(s,a)=a$, the element $\psi_{t,s}(a)$ can be seen as a map which carries 
$a$ to $y_t(s,a)$. 
Flows are related to dynamical systems. They differ from solutions. 
One of their interest lies in their characterization
as lipeomorphims (Lipschitz functions with a Lipschitz inverse), diffeomorphisms...

In this article, we develop a generic framework to construct flows from approximations.
We do not focus on a particular form of the solutions, so that our construction
is a \emph{non-linear sewing lemma}, modelled after the additive and multiplicative
sewing lemmas developed in \cite{lyons98a,feyel,gub04,friz14a}. 

In this first part, we study flows under weak conditions
and prove existence of a measurable flow even when the solutions of RDE
are not necessarily unique. This is based on a selection theorem 
\cite{cardona_semiflow2} due to J.E.~Cardona and L.~Kapitanski.
Such a result is new in the literature where
existence of flows was only proved under stronger regularity conditions
(the many approaches are summarized in~\cite{coutin-lejay2}). Besides, 
our approach also contains the 
additive and multiplicative sewing lemmas~\cite{feyel,coutin}. 
The rough equivalent of the Duhamel formula for solving 
linear RDE \cite{coutin-lejay1} with a perturbative terms follows 
directly from our construction. 

In a second part \cite{brault2}, we provide conditions
for uniqueness and continuity. Besides, we show that our construction 
encompass many of the previous
approaches or results: A.M.~Davie \cite{davie05a}, 
I.~Bailleul \cite{bailleul12a,bailleul13b,bailleul14b} and P.~Friz \& N.~Victoir \cite{friz}.

Our starting point in the world of classical analysis is the 
\emph{product formula} which relates how the iterated product of an approximation of a flow
converge to the flow \cite{abraham,chorin}. It is important both from the theoretical 
and numerical point of view.

On a Banach space $\uV$, let us consider a family $(\epsilon,a)\in\RR_+\times\uV \mapsto \Phi(\epsilon,a)$, 
called an \emph{algorithm} of class~$\cC^1$ in $(\epsilon,a)$ such that $\Phi(0,a)=a$.
The parameter $\epsilon$ is related to the quality of the approximation.

The algorithm $\Phi$ is \emph{consistent} with a vector field $f$ when 
\begin{equation}
    \label{eq:fp:1}
    f(a)=\frac{\partial \Phi}{\partial\epsilon}(0,a) ,\ \forall a\in\uV.
\end{equation}
For a consider algorithm, when $\phi_t(a)$ is the solution to $\phi_t(a)=a+\int_0^t f(\phi_s(a))\vd s$, 
\begin{equation}
    \label{eq:fp:2}
    \underbrace{\Phi(t/n,\Phi(t/n,\dotsb\Phi(t/n,a)))}_{n\text{ times}}
    \text{ converges to }\phi_t(a)\text{ as }n\to\infty.
\end{equation}
Eq.~\eqref{eq:fp:2} is called the product formula.

The Euler scheme for solving ODE is the prototypical example of such behavior.
Set $\Phi(\epsilon,a):=a+f(a)\epsilon$ so that \eqref{eq:fp:1} holds.
In this case, \eqref{eq:fp:2} expresses the convergence of the Euler scheme. 

Using the product formula, one recovers easily the proof
of Lie's theorem on matrices: If $A$ is a matrix, 
then $\exp(tA)$ is the solution to $\dot{Y}_t=AY_t$ with $Y_0=\mathrm{Id}$ and then  
\begin{equation*}
    \exp(tA)=\lim_{n\to\infty} \left(\mathrm{Id}+\frac{t}{n}A\right)^n.
\end{equation*}

For two matrices $A$ and $B$, $\exp(t(A+B))$ is given by
\begin{equation*}
    \exp(t(A+B))=\lim_{n\to\infty}\left(\exp\left(\frac{tA}{n}\right)
    \exp\left(\frac{tB}{n}\right)\right)^n.
\end{equation*}
To prove the later statement, we consider $\Phi(\epsilon,a)=\exp(\epsilon A)\exp(\epsilon B)a$
and we verify that $\partial_\epsilon \Phi(0,a)=(A+B)a$ for any matrix $a$.

For unbounded operators, it is also related to Chernoff and Trotter's 
results on the approximation of semi-groups \cite{engel,MR3592993}. 
The product formula also justifies the construction of some splitting schemes
\cite{MR2494199}.

In this article, we consider as an algorithm a family $\Set{\phi_{t,s}}_{0\leq s\leq t\leq T}$
of functions from~$\uV$ to~$\uV$ which is close to the identity map in short
time and such that $\phi_{t,s}\circ\phi_{s,r}$ is close to $\phi_{t,r}$
for any time $s\leq r\leq t$. For a path $x$ of finite $p$-variation, $1\leq p<2$
with values in $\RR^d$
and a smooth enough function $f:\RR^m\to\uL(\RR^d,\RR^m)$, 
such an example is given by $\phi_{t,s}(a)=a+f(a)x_{s,t}$. 

We then study the behavior of 
the composition $\phi^\pi$ of the $\phi_{t_{i+1},t_i}$ along a
partition $\pi=\Set{t_i}_{i=0,\dotsc,n}$. Clearly, as the mesh
of the partition $\pi$ goes to $0$, the limit, when it exists 
is a candidate to be a flow. In the example given above, it will be the flow
associated to the family of Young differential equations 
$y_t(a)=a+\int_0^t f(y_s(a))\vd x_s$, which means according to A.M.~Davie \cite{davie05a} that 
\begin{equation}
    \label{eq:intro:1}
    \abs{y_t(a)-\phi_{t,s}(y_s(a))}\leq L(a)\varpi(\omega_{s,t}).
\end{equation}
We show that measurable flows may exist for Young or Rough Differential Equations 
even when several paths satisfying  \eqref{eq:intro:1} exist.

In \cite{brault2,brault3}, we exhibit a condition on almost flow 
that ensure existence of Lipschitz flows. Such an almost flow is called a \emph{stable almost flow}.
Besides, we study further
the connection between almost flows and solutions in the sense of~\eqref{eq:intro:1}. 
In particular, when an almost flow is stable, solutions exist and are unique. 
Stronger convergence rate of numerical approximations, as well as continuity results are then given.

In order to present our main results, we introduce 
some necessary notations as well as some central notions such as galaxies.

\subsection{Notations, definitions and concepts}
\label{sec:notations}

The following notations and hypotheses will be constantly used throughout all this article.

\subsubsection{Hölder and Lipschitz continuous functions}

Let $\uV$ and $\uW$ be two metric spaces.

The space of continuous functions from $\uV$ to $\uW$ 
is denoted by $\cC(\uV,\uW)$. 

Let $d$ (resp. $d'$) be a distance on $\uV$ (resp. $\uW$).
For $\gamma\in \itvoc{0,1}$, 
we say that a function $f:\uV\to\uW$ is $\gamma$-\emph{Hölder} if 
\begin{equation*}
    \normf{\gamma}{f}\eqdef\sup_{\substack{a,b\in\uV,\\a\neq b}} \frac{d'(f(a),f(b))}{d(a,b)^\gamma}<+\infty.
\end{equation*}
If $\gamma=1$ we say that $f$ is \emph{Lipschitz}. We then set 
$\normlip{f}\eqdef \normf{1}{f}$.

For any integer $r\geq 0$ and $\gamma\in \itvoc{0,1}$, we denote by
$\cC^{r+\gamma}(\uV,\uW)$ the subspace of functions from $\uV$ to $\uW$ whose
derivatives $\dd^k f$ of order $k\leq r$ are continuous and such that
$\dd^r f$ is $\gamma$-Hölder.

We denote by $\cCb^{r+\gamma}(\uV,\uW)$   the subset of 
$\cC^{r+\gamma}(\uV,\uW)$ of bounded functions with 
bounded derivatives up to order $r$.

\subsubsection{Controls and remainders}

From now, $\uV$ is a topological, complete metric space with a distance $d$.
A distinguished point of $\uV$ is denoted by $0$.

We fix $\gamma\in \itvoc{0,1}$. 
Let $N_\gamma:\uV\to \itvco{1,+\infty}$ be a $\gamma$-Hölder continuous function with constant~$\normhold{\gamma}{N}$.
The index $\gamma$ in $N_\gamma$  refers to its regularity. If $N_\gamma$ is Lipschitz
continuous ($\gamma=1$), then we simply write $N$.

Let us fix a \emph{time horizon} $T$. We write $\TT:=[0,T]$ as well as 
\begin{equation}
    \label{eq:TT:1}
    \rTT_2:=\Set{(s,t)\in\TT^2\given s\leq t}
    \text{ and }
    \rTT_3:=\Set{(r,s,t)\in\TT^3\given r\leq s\leq t}.
\end{equation}

A \emph{control} $\omega:\rTT_2\to\RR_+$ is a \emph{super-additive} family, \textit{i.e.},
\begin{equation*}
    \omega_{r,s}+\omega_{s,t}\leq \omega_{r,t},\ \forall (r,s,t)\in\rTT_3
\end{equation*}
with $\omega_{s,s}=0$ for all $s\in\TT$, and for any $\delta>0$, there exists $\epsilon>0$
such that $\omega_{s,t}<\delta$ whenever $0\leq s\leq t\leq s+\epsilon$.
A typical example of a control is $\omega_{s,t}=C\abs{t-s}$ for a constant $C\geq 0$.

For $p\geq 1$, we say that a path $x\in\cC(\TT,\uV)$ is
a \emph{path of finite $p$-variation controlled by $\omega$} if
\begin{align*}
\normp{x}\eqdef 
\sup_{\substack{(s,t)\in\rTT_2,\\s\neq t}} 
\frac{d(x_s,x_t)}{\omega_{s,t}^{1/p}}<+\infty.
\end{align*}

A \emph{remainder} is a function $\varpi:\RR_+\to \RR_+$  which is continuous,
increasing and such that 
for some $0<\varkappa<1$, 
\begin{equation}
    \label{eq:h4}
    2\varpi\left(\frac{\delta}{2}\right)\leq \varkappa\varpi(\delta),\ \forall \delta>0.
\end{equation}
A typical example for a remainder is $\varpi(\delta)=\delta^\theta$ for any $\theta>1$. 

We consider that $\delta:\RR_+\to \RR_+$ is a non-decreasing function with $\lim_{T\to 0} \delta_T=0$.

Finally, let $\eta: \RR_+\to \RR_+$ be a continuous, increasing function such that for all $(s,t)\in\rTT_2$,
\begin{equation}
\label{eq:h_eta}
\eta(\omega_{s,t})\varpi(\omega_{s,t})^\gamma\leq \delta_T\varpi(\omega_{s,t}).
\end{equation}

Partitions of $\TT$ are  customary denoted by $\pi=\Set{t_i}_{i=0,\dotsc,n}$.
The \emph{mesh} $\abs{\pi}$ of a partition $\pi$ is 
$\abs{\pi}:=\max_{i=0,\dotsc,n} (t_{i+1}-t_i)$.
The discrete simplices $\pi^+_2$ and $\pi^+_3$ are defined similarly to
$\rTT_2$ and $\rTT_3$ in~\eqref{eq:TT:1}.

\subsubsection{Galaxies}
\label{sec:galaxy}

\begin{notation}
    We denote by $\cF(\uV)$ the set of functions
    $\Fam{\phi_{t,s}}_{(s,t)\in\rTT_2}$ from $\uV$ to $\uV$
    which are continuous in $(s,t)$, i.e. for any 
    $a\in\uV$, the map $(s,t)\in\rTT_2\mapsto \phi_{t,s}(a)$ is continuous.
\end{notation}
\begin{notation}[Iterated products]
    For any $\phi\in\cF(\uV)$, any partition $\pi$ of $\TT$
    and any $(s,t)\in\rTT_2$, we write 
\begin{equation}
    \label{eq:20}
    \phi_{t,s}^\pi\eqdef \phi_{t,t_j}\circ \phi_{t_j,t_{j-1}}
    \circ\dotsb\circ \phi_{t_{i+1},t_i}\circ \phi_{t_i,s},
\end{equation}
where $\itvcc{t_i,t_j}$ is the biggest interval of such kind contained in $\itvcc{s,t}\subset\TT$
(possibly, $t_i=t_j$). If no such interval exists, then $\phi^\pi_{t,s}=\phi_{t,s}$. 
\end{notation}

Clearly, for any partition, $\phi^\pi\in\cF(\uV)$. A trivial but important remark is that 
from the very construction, 
\begin{equation}
    \label{eq:30}
    \phi_{t,r}^\pi=\phi_{t,s}^\pi\circ \phi^\pi_{s,r}\text{ for any }s\in\pi,\ 
    r\leq s\leq t.
\end{equation}
In particular, $\Fam{\phi^\pi_{t,s}}_{(s,t)\in\pi^2_+}$ enjoys a (semi-)flow 
property when the times are restricted to the elements of $\pi$.

The article is mainly devoted to study the possible limits of
$\phi^\pi$ as the mesh of $\pi$ decreases to $0$.

\begin{notation}
    \label{not:norms}
    From a distance $d$ on $\uV$,
    we define the distance~$\Delta_{N_\gamma}$ 
    on the space of functions from $\uV$ to $\uV$ by 
    \begin{equation*}
	\Delta_{N_\gamma}(f,g)\eqdef\sup_{a\in\uV}\frac{d(f(a),g(a))}{N_\gamma(a)},
    \end{equation*}
where $N_\gamma$ is defined in Section~\ref{sec:notations}.

    This distance is extended on $\cF(\uV)$ by 
    \begin{equation*}
	\Delta_{N_\gamma,\varpi}(\phi,\psi)\eqdef\sup_{\substack{(s,t)\in\rTT_2\\ s\not=t}}
\frac{\Delta_{N_\gamma}(\phi_{t,s},\psi_{t,s})}{\varpi(\omega_{s,t})},
    \end{equation*}
where $\omega$, $\varpi$ are defined in Section~\ref{sec:notations}. 
The distance $\Delta_{N_\gamma,\varpi}$ may take infinite values.
\end{notation}

If $d$ is a distance for which $(\uV,d)$ is complete, then 
$(\cC(\uV,\uV),\Delta_{N})$ and
$(\cF(\uV),\Delta_{N,\varpi})$ are complete.

\begin{definition}[Galaxy]
\label{def_galaxy}
    We define the equivalence relation $\sim$ on $\cF(\uV)$ by 
    $\phi\sim\psi$ if and only if there exists a constant $C$ such that 
    \begin{equation*}
    d(\phi_{t,s}(a),\psi_{t,s}(a))\leq CN_\gamma(a)\varpi(\omega_{s,t}),
    \ \forall a\in\uV,\ \forall (s,t)\in\rTT_2.    
    \end{equation*}
    In other words, $\phi\sim\psi$ if and only if $\Delta_{N_\gamma,\varpi}(\phi,\psi)<+\infty$. 
    Each quotient class of $\cF(\uV)/\sim$ is called a \emph{galaxy}, which 
    contains elements of $\cF(\uV)$ which are at finite distance from each others.
\end{definition}


\subsection{Summary of the main results}

\label{sec:main:results}

The galaxies partition the space $\cF(\uV)$. Each galaxy may
contain two classes of elements on which we focus on this article:
\begin{enumerate}[topsep=-\parskip,noitemsep,leftmargin=1.5em]
    \item The \emph{flows}, that is the families of maps $\psi:\uV\to\uV$ 
	which satisfy 
	\begin{equation}
	    \label{eq:sum:1}
	    \psi_{t,r}=\psi_{t,s}\circ \psi_{s,r}, \ \forall (r,s,t)\in\rTT_3, 
    \end{equation}
    or equivalently, $\psi^\pi=\psi$
    (See \eqref{eq:20}) for any partition $\pi$.
\item The \emph{almost flows} which we see as time-inhomogeneous
    algorithms. Besides some conditions on the continuity and the growth
    given in Definition~\ref{def:almost_flow} below, 
    an almost flow $\phi$ is close to a flow
    with the difference that 
    \begin{equation*}
	d(\phi_{t,s}\circ \phi_{s,r}(a),\phi_{t,r}(a))\leq N_\gamma(a)\varpi(\omega_{r,t}),
	\ \forall (r,s,t)\in\rTT_3,\ a\in\uV,
    \end{equation*}
    for a suitable function $N_\gamma:\uV\to\itvco{1,+\infty}$.
\end{enumerate}

Along with an almost flow $\phi$ comes the notion of \emph{manifold
of D-solutions}, that is a family $y:=\Set{y_{t}(a)}_{t\geq 0,\ a\in\uV}$
of paths such that
\begin{equation}
  \label{eq:intro:davie}
    d(y_t(a),\phi_{t,s}(y_s(a)))\leq C\varpi(\omega_{s,t}),\ \forall (s,t)\in\rTT_2.
\end{equation}
Each path $y(a)$ that satisfies \eqref{eq:intro:davie} is called a \emph{D-solution}.
This definition expands naturally the one introduced by A.M.~Davie in \cite{davie05a}. 

Clearly, a manifold of D-solutions associated to an almost flow $\phi$
is also associated to any almost flow in the same galaxy as $\phi$.
Besides, if a flow $\psi$ exists in the same galaxy as $\phi$, 
then $z_t(a)=\psi_{t,0}(a)$ defines a manifold of D-solutions~$z$.
Flows are more constrained objects than solutions as \eqref{eq:sum:1}
implies some compatibility conditions, while it is possible to create
new D-solutions by splicing two different ones.
As it will be shown in \cite{brault2}, uniqueness of manifold of D-solutions
is strongly related to existence of a flow. 

Given an almost flow $\phi$, it is natural to study the
limit of the net $\Set{\phi^\pi}_{\pi}$ as the mesh 
of $\pi$ decreases to $0$. Any limit will be a good candidate
to be a flow. 

Our first main result (Theorem~\ref{thm:1}) asserts that for any almost flow $\phi$ 
in a galaxy $G$, any iterated map $\phi^\pi$ belongs to $G$ whatever the partition $\pi$ although the map $\phi^\pi$
is not necessarily an almost flow. More precisely, one controls 
$\Delta_{N_\gamma,\varpi}(\phi^\pi,\phi)$ uniformly in the partition~$\pi$.

An immediate corollary is that any possible limit of $\Set{\phi^\pi}_{\pi}$ 
 as the mesh of the partition decreases to $0$ also belongs to $G$.

Our second main result (Theorem~\ref{thm:nls:1}) is that when the underlying Banach
space $\uV$ is finite-dimensional, 
there exists at least one measurable flow in a galaxy containing
an almost flow, even when several manifolds of D-solutions may exist. 
Our proof uses a recent result of J.E.~Cardona and L.~Kapitanski
\cite{cardona_semiflow2} on selection theorems.

Our third result is to give conditions ensuring the existence
of at most one flow in a galaxy. A sufficient condition is given for 
the galaxy $G$ contains a Lipschitz flow~$\psi$. In this case, 
 $\Set{\phi^\pi}_{\pi}$ converges to~$\psi$ whatever the
 almost flow $\phi$ in $G$. The rate of convergence is also quantified. 

Finally, we apply our results to the additive, multiplicative sewing 
lemmas \cite{feyel} as well as to the algorithms proposed by A.M. Davie
in \cite{davie05a} to show existence of measurable flows even without 
uniqueness. In the sequel \cite{brault2}, we study in details
the properties of Lipschitz flows and give some conditions
on almost Lipschitz flows to generate them. In addition, 
we also apply our results to other approximations of RDE, namely
the one proposed by P.~Friz \& N.~Victoir \cite{friz}
and the one proposed by I.~Bailleul \cite{bailleul12a} by using perturbation arguments.

\subsection{Outline}

We show in Section~\ref{sec:almost_flow} that a uniform control of the iterated
product of approximation of flows with respect to the subdivision.  In
Section~\ref{sec:weak_sewing_lemma}, we prove our main result: the existence of
a measurable flow under weak conditions of regularity.  Then, in
Section~\ref{sec:lipschitz} we show the existence and uniqueness of a Lipschitz
flow under stronger assumptions.  Moreover, we give a rate of the convergence
of the iterated product to the flow.  In Section~\ref{sec:perturbation}, we
show that additive perturbations preserve 
the convergence of iterated products of approximations of flows.
Finally, we recover
in Section~\ref{sec:applications} the additive \cite{lyons98a}, multiplicative
\cite{feyel,coutin-lejay1} and Davie's sewing lemmas~\cite{davie05a}.

\section{A uniform control over almost flows}
\label{sec:almost_flow}

In this section, we define almost flows which serve as approximations. 
The properties of an almost flow $\phi$ are weaker than the minimal condition
to obtain the convergence of the iterated product $\phi^\pi$ as the mesh of
the partitions $\pi$ decreases to $0$.  However, we prove in
Theorem~\ref{thm:1} that we can control $\phi^\pi$ uniformly over the
partitions~$\pi$.  This justifies our definition.

\subsection{Definition of almost flows}

\begin{definition}[Almost flow]
\label{def:almost_flow}
An element $\phi\in\cF(\uV)$ is an \emph{almost flow} if for 
any $(r,s,t)\in\rTT_3$, $a\in\uV$, 
\begin{gather}
    \label{eq:h0}
    \phi_{t,t}(a)=a,\\
    \label{eq:h1}
    d(\phi_{t,s}(a),a)\leq \delta_{T} N_\gamma(a),\\
    \label{eq:h2}
    d(\phi_{t,s}(a),\phi_{t,s}(b))\leq (1+\delta_{T})d(a,b)+\eta(\omega_{s,t})d(a,b)^\gamma,\\
    \label{eq:h3}
    d(\phi_{t,s}\circ\phi_{s,r}(a),\phi_{t,r}(a))\leq N_\gamma(a)\varpi(\omega_{r,t}).
\end{gather}
\end{definition}

\begin{remark}
    A family $\phi\in\cF(\uV)$ satisfying condition \eqref{eq:h2} with $\gamma=1$
    is said to be \emph{quasi-contractive}. This notion plays an important
    role in the fixed point theory \cite{berinde04}.
\end{remark}

\begin{definition}[Iterated almost flow]
    For a partition $\pi$ and an almost flow $\phi$, 
    we call $\phi^\pi$ an \emph{iterated almost flow}, where
    $\phi^\pi$ is the iterated product defined in \eqref{eq:20}.
\end{definition}

\begin{definition}[A flow] 
    A \emph{flow} $\psi$ is a family of functions $\Set{\psi_{t,s}}_{(s,t)\in\rTT_2}$ from $\uV$ to $\uV$ (not necessarily continuous in $(s,t)$) such that
    $\psi_{t,t}(a)=a$ and $\psi_{t,s}\circ\psi_{s,r}(a)=\psi_{t,r}(a)$ for any $a\in\uV$ and $(r,s,t)\in\rTT_3$. 
\end{definition}

In this paper, we consider almost flows which are continuous although flows may \textit{a priori} be discontinuous
(See Theorem~\ref{thm:nls:1}).


\subsection{A uniform control on iterated almost flows}

Before proving our main result in Section~\ref{sec:weak_sewing_lemma},
we prove an important uniform control over $\phi^\pi$.

\begin{theorem}
    \label{thm:1}
    Let $\phi$ be an almost flow. Then there exists
    a time horizon $T$ small enough and constants $L_T\geq 1$ as well 
    as $K_T\geq 1$ that decrease to $1$ as~$T$ decreases to $0$ such 
    that 
    \begin{align}
	d(\phi^\pi_{t,s}(a),\phi_{t,s}(a))&\leq L_TN_\gamma(a)\varpi(\omega_{s,t}),\\
	N_\gamma(\phi^\pi_{t,s}(a))&
	\leq K_T N_\gamma(a),\label{eq:thm1_KT} 
    \end{align}
    for any $(s,t)\in\TT_+^2$, $a\in\uV$ and any  partition $\pi$ of $\TT$.
    The choice of $T$, $L_T$ and $K_T$ depend only on $\delta$, $\varpi$, $\omega$, $\gamma$ and $\normhold{\gamma}{N_\gamma}$.
\end{theorem}

\begin{remark} The distance $d$ may be replaced by a pseudo-distance
    in the statement of Theorem~\ref{thm:1}.
\end{remark}

The proof of Theorem~\ref{thm:1} is inspired by 
the one of the Claim in the proof of Lemma~2.4 in \cite{davie05a}. 
With respect to the one of A.M.~Davie, we consider obtaining 
a uniform control over a family of elements indexed by $(s,t)\in\rTT_2$
which are also parametrized by points in $\uV$. 

\begin{definition}[Successive points / distance between two points]
    Let $\pi$ be a partition of $\itvcc{0,T}$. Two points $s$ and $t$ of $\pi$
are said to be \emph{at distance $k$} if there are exactly $k-1$ points between them in $\pi$. 
We write $d_\pi(t,s)=k$. Points at distance~$1$ are called \emph{successive points} in $\pi$. 
\end{definition}
\begin{proof} 
    For $a\in\uV$, $(r,t)\in\rTT_2$ and a partition $\pi$, we set 
    \begin{equation*}
	U_{r,t}(a):=d(\phi^\pi_{t,r}(a),\phi_{t,r}(a)).
    \end{equation*}

    We now restrict ourselves to the case $(r,t)\in\pi^+_2$.
    To control $U_{r,t}(a)$ in a way that does not depend on $\pi$, 
    we use an induction in the distance between~$r$ and~$t$.

    Our induction hypothesis is that there exist constants $L_T\geq 0$ and $K_T\geq 1$  
    independent from the partition $\pi$ such that for any $(r,t)\in\pi^+_2$ at distance at most~$m\geq 1$, 
    \begin{align}
	\label{eq:induction:1}
	U_{r,t}(a)&\leq L_TN_\gamma(a)\varpi(\omega_{r,t}),\\
	\label{eq:induction:2}
	N_\gamma(\phi^\pi_{t,s}(a))&\leq K_T N_\gamma(a),
    \end{align}
with $K_T$ decreases to $1$ at $T$ goes to $0$.

    The induction hypothesis is true for $m=0$, since $U_{r,r}(a)=0$ and $N_\gamma(\phi^\pi_{r,r}(a))=N_\gamma(a)$.

    If $r$ and $t$ are successive points, $\phi^\pi_{t,r}=\phi_{t,r}$ so 
    that $U_{r,t}(a)=0$. Thus, \eqref{eq:induction:1} is true for $m=1$. 
    With \eqref{eq:h1} and since $N_\gamma(a)^\gamma\leq N_\gamma(a)$ as $N_\gamma(a)\geq 1$
    by hypothesis,
    \begin{multline*}
	N_\gamma(\phi_{t,s}(a))\leq N_\gamma(\phi_{t,s}(a))-N_\gamma(a)
	+N_\gamma(a)
	\leq \normhold{\gamma}{N_\gamma}d(\phi_{t,s}(a),a)^\gamma+N_\gamma(a)\\
	\leq \normhold{\gamma}{N_\gamma}\delta_T^\gamma N_\gamma(a)^\gamma+N_\gamma(a)
	\leq (\normhold{\gamma}{N_\gamma}\delta_T^\gamma+1)N_\gamma(a).
    \end{multline*}

    For $m=2$, it is also true with $L_T= 1$
    as $U_{r,t}(a)=d(\phi_{t,s}\circ\phi_{s,r}(a),\phi_{t,r}(a))$ where $s$
    is the point in the middle of $r$ and $t$. This proves \eqref{eq:induction:1}.

    In addition, using \eqref{eq:induction:1}, 
    \begin{multline}
	\label{eq:34}
	N_\gamma(\phi^\pi_{t,s}(a))\leq \normhold{\gamma}{N}d(\phi^\pi_{t,s}(a),\phi_{t,s}(a))^\gamma
	+N_\gamma(\phi_{t,s}(a))-N_\gamma(a)+N_\gamma(a)\\
	\leq K_T N_\gamma(a)
	\text{ with }K_T:=\normhold{\gamma}{N_\gamma}\varpi(\omega_{0,T})^\gamma
	+ \normhold{\gamma}{N_\gamma}\delta_T^\gamma+1.
    \end{multline}
   Clearly, $K_T$ decreases to $1$ at $T$ decreases to $0$.
    This proves \eqref{eq:induction:2} whenever \eqref{eq:induction:1}, 
    in particular for $m=1$. 

    Assume that \eqref{eq:induction:1}-\eqref{eq:induction:2} when the distance
    between~$r$ and~$t$ is smaller than $m$ for some $m\geq 2$. 

    For $s\in\pi$, $r\leq s\leq t$, with \eqref{eq:30}, 
    \begin{multline*}
	U_{r,t}(a)\leq d(\phi^\pi_{t,s}\circ\phi^\pi_{s,r}(a),\phi_{t,s}\circ\phi^\pi_{s,r}(a))\\
	+d(\phi_{t,s}\circ\phi^\pi_{s,r}(a),\phi_{t,s}\circ\phi_{s,r}(a))
	+d(\phi_{t,s}\circ\phi_{s,r}(a),\phi_{t,r}(a)).
    \end{multline*}
    With \eqref{eq:h2} and \eqref{eq:h3},
    \begin{equation}
	\label{eq:31}
	U_{r,t}(a)
	\leq
	U_{s,t}(\phi^\pi_{s,r}(a))
	+
	(1+\delta_T) U_{r,s}(a)+\eta(\omega_{s,t})U_{r,s}(a)^\gamma
	+
	N_\gamma(a)\varpi(\omega_{r,t}).
    \end{equation}
    Using \eqref{eq:31} on $U_{s,t}(\phi^\pi_{s,r}(a))$ 
    by replacing $(r,s,t)$ by $(s,s',t)$, where $s$ and $s'$ are two successive points of $\pi$
    leads to 
    \begin{equation*}
	U_{s,t}(\phi^\pi_{s,r}(a))
	\leq
	U_{s',t}(\phi^\pi_{s',r}(a))
	+
	 N_\gamma(\phi_{s,r}^\pi(a))\varpi(\omega_{s,t})
    \end{equation*}
    since $\phi_{s',s}\circ\phi^\pi_{s,r}(a)=\phi_{s',r}^\pi(a)$ and $U_{s,s'}(a)=0$.

    With \eqref{eq:induction:1}-\eqref{eq:induction:2} and our hypothesis on $\eta$, 
    again since $N_\gamma\geq 1$ and $s,s'$ are at distance less than $m$
    provided that $s$ is at distance at most $2$ from $t$, 
    \begin{equation}
	\label{eq:33}
	U_{r,t}(a)
	    \leq N_\gamma(a)\left(K_T L_T \varpi(\omega_{s',t})
	    +(L_T+\delta_T(1+L_T^\gamma))  \varpi(\omega_{r,s})
	    +(1+K_T) \varpi(\omega_{r,t})\right).
\end{equation}
This inequality also holds true for $r=s$ or $s'=t$.

We proceed as in \cite{davie05a} to split $\pi$ in \textquote{essentially}
    two parts. 
    We set 
    \begin{equation*}
	s'\eqdef\min\Set*{\tau\in\pi\given\omega_{r,\tau}\geq \frac{\omega_{r,t}}{2}}
	\text{ and }
	s\eqdef\max\Set*{\tau\in\pi\given \tau<s'}.
    \end{equation*}
    Hence, $s$ and $s'$ are successive points with $r\leq s<s'\leq t$
    and $\omega_{r,s}\leq \omega_{r,t}/2$.
    Besides, 
    since $\omega$ is super-additive, $\omega_{r,s'}+\omega_{s',t}\leq \omega_{r,t}$. 
    Therefore, 
    \begin{equation}
	\label{eq:38}
	\omega_{r,s}\leq \frac{\omega_{r,t}}{2}\text{ and }
	\omega_{s',t}\leq \frac{\omega_{r,t}}{2}.
    \end{equation}
    With such a choice, since $L_T\geq 1$, \eqref{eq:33} becomes with \eqref{eq:h4}:
    \begin{equation*}
	U_{r,t}(a)\leq \left(L_T\frac{K_T+1+2\delta_T}{2}\varkappa +1+K_T\right)
	N_\gamma(a)\varpi(\omega_{r,t}). 
    \end{equation*}
    If $T$ is small enough so that 
    \begin{equation*}
	\alpha:=\frac{1+K_T+2\delta_T}{2}\varkappa<1\text{ and }L_T\alpha+1+K_T\leq L_T, 
    \end{equation*}
    then one may choose $L_T$ so that $L_T\alpha+1+K_T\leq L_T$, that is $L_T\geq \max\Set{1,(1+K_T)/(1-\alpha)}$.
    This choice ensures that $U_{r,t}(a)\leq L_TN_\gamma(a)\varpi(\omega_{r,t})$
    when $r$ and $t$ are at distance $m$. Condition \eqref{eq:induction:2}
    follows from \eqref{eq:34} and \eqref{eq:induction:1}. 

    The choice of $T$, $K_T$ and $L_T$ does not depend on $\pi$. In particular, 
    $d(\phi^\pi_{t,s}(a),\phi_{t,s}(a))\leq L_TN_\gamma(a)\varpi(\omega_{s,t})$
    becomes true for any $(s,t)\in\rTT_2$ (it is sufficient to add 
    the points $s$ and $t$ to $\pi$). 
\end{proof}
\begin{corollary}
Let $\phi$ be an almost flow and $\pi$ be a partition of $\TT$. Then $\phi^\pi\sim\phi$
(we have not proved that $\phi^\pi$ is itself an almost flow).
\end{corollary}
\begin{notation}
\label{not:Sphi}
    For an almost flow $\phi$, let us denote by $S_\phi(s,a)$ 
    the set of  all the possible limits of the net $\Set{\phi^\pi_{\cdot,s}(a)}_{\pi}$
    in $(\cC([s,T],\uV),\normsup{\cdot})$ for nested partitions. 
\end{notation}

When $\uV$ is a finite dimensional space with the norm 
$\abs{\cdot}$, $S_\phi(s,a)\ne\emptyset$. We start with a lemma 
which will be useful to prove some equi-continuity.
We denote $\bar{B}(0,R)$ the closed ball centered in $0$ of radius
$R> 0$.

\begin{lemma}
\label{lem:unif:cont}
Let $R>0$. We assume that $\uV$ is a finite dimensional vector space and $\phi$ is an almost flow.  Then, for all $\epsilon>0$ there exists $\delta>0$ such that
for all $(s,t),(s',t')\in\rTT_2$ with $|t-t'|,|s-s'|<\delta$ and $a\in \bar{B}(0,R)$,
\begin{align*}
\abs{\phi_{t,s}(a)-\phi_{t',s'}(a)}\leq \epsilon.
\end{align*}
\end{lemma}
\begin{proof}
In all the proof $(s,t),(s',t')\in\rTT_2$.
For any $a\in \bar{B}(0,R)$, $(s,t)\in\rTT_2\mapsto \phi_{t,s}(a)$ is continuous, so  uniformly continuous on the compact $\rTT_2$. Let $\epsilon>0$, there is $\delta_a$ such that for all $|t-t'|,|s-s'|<\delta_a$,
\begin{align}
\label{eq:lem:unif1}
\abs{\phi_{t,s}(a)-\phi_{t',s'}(a)}\leq \frac{\epsilon}{3}.
\end{align}
For a $\rho>0$ with \eqref{eq:h2}, for all $b\in B(a,\rho)$,
\begin{align*}
\abs{\phi_{t,s}(a)-\phi_{t,s}(b)}\leq (1+\delta_T)\rho+\eta(\omega_{0,T})
\rho^\gamma,
\end{align*}
where $B(a,\rho)$ denotes the open ball centred in $a$ of radius $\rho$.
We choose $\rho(\epsilon)>0$  such that $(1+\delta_T)\rho(\epsilon)+\eta(\omega_{0,T})
\rho(\epsilon)^\gamma\leq \epsilon/3$ to obtain for all $b\in B(a,\rho(\epsilon))$,
\begin{align}
\label{eq:lem:unif2}
\abs{\phi_{t,s}(a)-\phi_{t,s}(b)}\leq \frac{\epsilon}{3}.
\end{align}

 We note that $\bigcup_{a\in \bar{B}(0,R)}B(a,\rho(\epsilon))$ is a covering of $\bar{B}(0,R)$ which is a compact of $\uV$. There exist an integer $N$ and a finite family of balls $B(a_i,\rho(\epsilon))$ for $i\in\{1,\dots,N\}$ such that
$\bar{B}(0,R)\subset\bigcup_{i\in\{1,\dots,N\}}B(a_i,\rho(\epsilon))$.

It follows that for all $b\in\bar{B}(0,R)$ there exists $i\in\{1,\dots,N\}$ such that $b\in B(a_i,\rho(\epsilon))$. From \eqref{eq:lem:unif1}-\eqref{eq:lem:unif2} there exists $\delta_{a_i}>0$ such that for all  $|t-t'|,|s-s'|<\delta_{a_i}$,
\begin{align*}
\abs{\phi_{t,s}(b)-\phi_{t',s'}(b)}\leq \epsilon.
\end{align*}
Taking $\delta:=\min_{i\in\{1,\dots,N\}}\delta_{a_i}$, we obtain that
for all $|t-t'|,|s-s'|<\delta$ and $b\in\bar{B}(0,R)$,
\begin{align*}
\abs{\phi_{t,s}(b)-\phi_{t',s'}(b)}\leq \epsilon,
\end{align*}
which achieves the proof.
\end{proof}

\begin{proposition}
    \label{prop:sel:1}
    Assume that $\uV$ is finite-dimensional and that $\phi$ is an almost flow.
    Then $S_\phi(r,a)$ is not empty for any $(r,a)\in\TT\times\uV$.
\end{proposition}

\begin{proof}
    Let us show for an almost flow $\phi$, 
    $\Set{\phi^\pi_{\cdot,r}(a)}_{\pi}$ is equi-continuous and bounded. The result is 
    then a direct consequence of the Ascoli-Arzelà theorem.

Let $(\pi_n)_{n\in\NN}$ be an increasing sequence of partitions of $\TT$ such that
$|\pi_n|\rightarrow 0$ when $n\rightarrow +\infty$ and $\bigcup_{n\in\NN} \pi_n$ dense in $\TT$.
Let $R>0$ and $R':=N_\gamma(R)(L_T\varpi(\omega_{0,T})+\delta_T)+R$. From Theorem~\ref{thm:1}, for any $a\in\bar{B}(0,R)$ and $(s,t)\in \TT_2^+$,
\begin{multline*}
\abs{\phi^{\pi_n}_{t,s}(a)}\leq \abs{\phi_{t,s}^{\pi_n}(a)-\phi_{t,s}(a)}+
\abs{\phi_{t,s}(a)-a}+\abs{a}\\
\leq  N_\gamma(a)(L_T\varpi(\omega_{0,T})+\delta_T)+\abs{a}
\leq R'.
\end{multline*}     
Let $(r,s,t)\in \rTT_3$ and $s,t\in\bigcup_{n\in\NN}\pi_n$, let
$\epsilon>0$ and $\delta>0$ given by Lemma~\ref{lem:unif:cont} for~$\epsilon/2$. For $|t-s|<\delta$, let $m$ an integer such that $s,t\in\pi_m$.

We differentiate two cases. If $m\leq n$, then $\pi_m\subset\pi_n$, which implies that $\phi_{t,r}^{\pi_n}=\phi_{t,s}^{\pi_n}\circ\phi_{s,r}^{\pi_n}$. From Theorem~\ref{thm:1} and Lemma~\ref{lem:unif:cont}, for all $a\in\bar{B}(0,R)$, for all $|t-s|<\delta$
\begin{equation*}
\abs{\phi_{t,r}^{\pi_n}(a)-\phi_{s,r}^{\pi_n}(a)}\leq
\abs{\phi_{t,s}^{\pi_n}\circ\phi^{\pi_n}_{s,r}(a)-\phi^{\pi_n}_{s,r}(a)}
\leq \frac{\epsilon}{2}+L_TN_\gamma(\phi_{s,r}^{\pi_n}(a))\varpi(\omega_{s,t}).
\end{equation*}
Since, $\omega_{s,s}=0$ and $\omega$ is continuous close to diagonal,
there exists $\delta'>0$ such that for all $|t-s|<\delta'$,
 $L_TN_\gamma(\phi_{s,r}^{\pi_n}(a))\varpi(\omega_{s,t})<\epsilon/2$. Thus for all
 $|t-s|<\min{(\delta,\delta')}$, $\abs{\phi_{t,r}^{\pi_n}(a)-\phi_{s,r}^{\pi_n}(a)}\leq \epsilon$.
 
 In the case $m>n$, let $s_{-},s_{+}$ be successive points in $\pi_n$ such that $[s,t]\subset [s_{-},s_{+}]$. Then, 
 $\phi_{t,r}^{\pi_n}(a)=\phi_{t,s_{-}}\circ\phi_{s_{-},r}^{\pi_n}(a)$
 and $\phi_{s,r}^{\pi_n}(a)=\phi_{s,s_{-}}\circ\phi_{s_{-},r}^{\pi_n}(a)$.
According to Lemma~\ref{lem:unif:cont}, for $|t-s|<\delta$ with $s,t\in\pi_m$, and all $a\in\bar{B}(0,R)$,
\begin{align}
\abs{\phi_{t,r}^{\pi_n}(a)-\phi_{s,r}^{\pi_n}(a)}=\abs{\phi_{t,s_{-}}\circ\phi_{s_{-},r}^{\pi_n}(a)-\phi_{s,s_{-}}\circ\phi_{s_{-},r}^{\pi_n}(a)}
\leq \epsilon.\label{eq:unif_equi_cont}
\end{align}
By continuity of $t\mapsto\phi_{t,r}^{\pi_n}(a)$, and density of $\bigcup_{m\in\NN}\pi_m$ in $\TT$, we obtain \eqref{eq:unif_equi_cont} for all $(s,t)\in\rTT_2$ with $r\leq s$ and $\abs{t-s}<\delta$.
This proves that $\Set{t\mapsto \phi_{t,r}^{\pi_n}(a)}_n$ is uniformly equi-continuous for all $a\in\uV$.
We conclude the proof with the Ascoli-Arzelà theorem.
\end{proof}


\section{The non-linear sewing lemma}
\label{sec:weak_sewing_lemma}

We now show that in the finite dimensional case, we can build a flow from $\phi^\pi$
using a selection principle~\cite{cardona_semiflow2}.
In this section, we consider that almost flows $\phi$ are defined for $\TT:=\itvco{0,+\infty}$.

\begin{definition}[Solution in the sense of Davie or D-solution]
    \label{def:davie}
    For an almost flow $\phi$, a time $r\in\TT$ and a point $a\in\uV$, 
    a   \emph{solution in the sense of Davie} (or a D-solution) 
    is a path $y\in\cC(\intSS,\uV)$ with $\intSS=[r,r+T]\subset\TT$  such that
    \begin{equation}
    \label{eq:def_davie}
    d(y_t,\phi_{t,s}(y_s))\leq KN_{\gamma}(a)\varpi(\omega_{s,t}),\ \forall (s,t)\in\intSS_+^2,
  \end{equation}
  where $K\geq 0$ is a constant.
\end{definition}

\begin{remark}
Our definition of D-solutions extends the one of Davie in \cite{davie05a}
to a metric space $\uV$ and a general almost flow $\phi$.
\end{remark}

\begin{remark}
When $\phi$ is only an almost flow, it is not guaranteed that 
a D-solution exists or is unique. 
When $S_{\phi}(r,a)\ne\emptyset$ (see Notation~\ref{not:Sphi}), we prove below in Lemma~\ref{lem:sel:2}
that a D-solution exists.
However, even if for all $(r,a)\in \TT\times\uV$,
$S_\phi(r,a)\ne\emptyset$
this does not imply the existence of families of solutions  $\Set{\psi_{\cdot,r}(a)}_{r\in\TT,a\in\uV}$ which satisfies
the flow property.
\end{remark}

\begin{notation}
We denote $\Omega(r,a)$ the set of continuous paths such that  $y\in\cC(\intSS,\uV)$ verifies $y_r=a$.
We denote by $G^K_\phi(r,a)$ the set of paths
in $\Omega(r,a)$ verifying \eqref{eq:def_davie} for the constant $K$.
\end{notation}

\begin{definition}[Splicing of paths]
    \label{def:splicing}
    For $r\leq s$, let us consider $y(r,a)\in\Omega(r,a)$ and
    $z(s,b)\in\Omega(s,b)$ with $b=y_s(r,a)$.
    Their \emph{splicing} is 
    \begin{equation*}
	( y\bowtie_s z)_t(r,a)=\begin{cases}
	    y_t(r,a)&\text{ if }t\leq s,\\
	    z_t(s,y_s(r,a))&\text{ if }t\geq s.
	\end{cases}
    \end{equation*}
\end{definition}
We restate here, the definition of a family of abstract  integral local funnels (Definition~$2$ in \cite{cardona_semiflow2}), which leads to the existence of a measurable flow.

\begin{definition}
\label{def:local_funnel}
A family $F(r,a)$ with $r\in \itvco{0,+\infty}$, $a\in\uV$, will be called a 
\emph{family of abstract local integral funnels} with \emph{terminal time} $T(r,a)\in\itvoo{0,+\infty}$
if 
\begin{description}
    \item[H0] The map $(r,a)\in \itvco{0,+\infty}\times\uV\mapsto T(r,a)$ is  lower semi-continuous in the sense that if $(r_n,a_n)\rightarrow (r,a)$, then 
$T(r,a)\leq\liminf_n T(r_n,a_n)$.
\item[H1] Every set $F(r,a)$ is a non-empty compact in the space $\cC(\itvcc{r,{r+T(r,a)}},\uV)$ and 
every path $y(r,a)\in F(r,a)$ is a continuous map from $\itvcc{r,{r+T(r,a)}}$ to~$\uV$ with $y_r(r,a)=a$.
\item[H2] For all $r\geq 0$, the map $a\in\uV\mapsto F(r,a)$ is measurable in the sense that for any closed subset $\cK\subset\cC(\itvcc{0,1},\uV)$,
$\Set{a\in\uV\given\tilde{F}(r,a)\bigcap\cK\ne\emptyset}$ is Borel, where $\tilde{F}(r,a)$ is the set of re-parametrizations of paths of $F(r,a)$ on $\itvcc{0,1}$.
\item[H3] If $y(r,a)\in F(r,a)$ and $\tau< T(r,a)$, then
  $T(r+\tau,y_{r+\tau}(r,a))\geq T(r,a)-\tau$ and $t\in
  [r+\tau,r+\tau+T(r+\tau,y_{r+\tau}(r,a))]\mapsto y_t(r,a)$ belongs
  to $F(r+\tau,y_{r+\tau}(r,a))$.
\item[H4] If $y(r,a)\in F(r,a)$ and $\tau< T(r,a)$, and $z\in
  F(r+\tau,y_{t+\tau}(r,a))$ then the spliced path
  (Definition~\ref{def:splicing}) $x:=y\bowtie_{r+\tau} z$ belongs to  $F(r,a)$.
\end{description}
\end{definition}

\begin{definition}[Lipschitz almost flow]
    \label{def:lipschitz:almost:flow}
    A \emph{Lipschitz almost flow} is an almost flow for
    which \eqref{eq:h2} is satisfied with $\eta=0$, and $N_\gamma=N_1=N$
    is Lipschitz continuous.
\end{definition}
A flow is constructed by assigning to each point of the space a particular D-solution,
in a sense which is compatible.

\begin{hypothesis}
\label{hypo:varpi}
Let $\uV$ be a finite dimensional vector space.
Let $\phi:=\Set{\phi_{t,s}}_{0\leq s\leq t<+\infty}$
be a Lipschitz almost flow (Definition~\ref{def:lipschitz:almost:flow}) with $N$ bounded. 
We fix a time horizon $T>0$ such that $\varkappa(1+\delta_T)<1$.
\end{hypothesis}

\begin{remark}
When $N$ bounded, we can choose $K_T=1$ and $L_T= 2/(1-\varkappa(1+\delta_T))$ 
where $K_T$ and $L_T$ are the constants of Theorem~\ref{thm:1}.
\end{remark}

The main theorem of this paper is the following one. 

\begin{theorem}[Non-linear sewing lemma, weak formulation]
    \label{thm:nls:1}
Under Hypothesis~\ref{hypo:varpi}, there exists 
    $\psi\in\cF(\uV)$ in the same galaxy as $\phi$ satisfying the flow property 
    and such that $\psi_{t,s}$ is Borel
    measurable for any $(s,t)\in\TT_+^2$. 
\end{theorem}

\begin{remark}
Proving such a result with a general Banach space $\uV$ is false
as even existence of solutions to ordinary differential equations may fail
\cite{dieudonne,hajek}.
\end{remark}

\begin{remark}
To prove Theorem~\ref{thm:nls:1}, we show that $(G^{L_T}(r,a))_{r\in
\itvco{0,+\infty},a\in\uV}$ is a family of abstract local integral funnels in the
sense of Definition~\ref{def:local_funnel}.  Then, we use Theorem~$2$ of
\cite{cardona_semiflow2}.
\end{remark}

Lemmas~\ref{lem:sel:T}-\ref{lem:sel:5} prove that $G^{L_T}(r,a)$ is a family of
local abstract funnels in the sense of the Definition~2 in
\cite{cardona_semiflow2}. Then we apply Theorem~2 in \cite{cardona_semiflow2}
to obtain the above theorem.

\begin{lemma}
\label{lem:sel:T}
Under  Hypothesis~\ref{hypo:varpi}, the terminal time $T(r,a):=T$ is
independent of the starting time $r$ and the starting point $a$.  
In particular, \textup{H0} holds for $F=G^{L_T}_\phi$.
\end{lemma}
\begin{proof}
It is sufficient to notice that the constants $\varkappa$ and $\delta_T$ do not depend on  $a\in\uV$ neither
on $r$.
\end{proof}

We recall that $S_\phi(r,a)$ is defined in  Notation~\ref{not:Sphi}.
Our first result is that when $S_\phi(r,a)\neq\emptyset$, then there exists
at least one D-solution in $G^{L_T}_\phi(r,a)$.  

\begin{lemma}
    \label{lem:sel:2}
    Assume that $K\geq K_TL_T$ in Definition~\ref{def:davie}, where $K_T$ and $L_T$ are constants in Theorem~\ref{thm:1}.
    For any $(r,a)\in\TT\times\uV$, $S_\phi(r,a)\subset G^K_\phi(r,a)$ for an almost flow $\phi$
    (note that $S_\phi(r,a)$ may be empty).
\end{lemma}
\begin{proof}
    If $y\in S_\phi(r,a)$ when $S_\phi(r,a)\neq\emptyset$,
    then there exists a sequence $\Set{\pi_k}_{k\in\NN}$
    of partitions such that $y_t=\lim\phi^{\pi_k}_{t,r}(a)$ uniformly in $t\in[r,T]$. We note that $y_r=a$.
    For $k\in\NN$ and  $s_k\in \pi_k$, with \eqref{eq:30}
    and Theorem~\ref{thm:1},
    \begin{align}
    \nonumber
	d(\phi^{\pi_k}_{t,r}(a),\phi_{t,s_k}\circ\phi^{\pi_k}_{s_k,r}(a))
	&=
	d(\phi^{\pi_k}_{t,s_k}\circ \phi^{\pi_k}_{s_k,r}(a),\phi_{t,s_k}\circ\phi^{\pi_k}_{s_k,r}(a))\\
	\nonumber
	&\leq L_TN_\gamma(\phi^{\pi_k}_{s_k,r}(a))\varpi(\omega_{s_k,t})\\
	&\leq L_TK_TN_\gamma(a)\varpi(\omega_{s_k,t}).\label{eq:conv_pi_k}
    \end{align}
    moreover, fixing $s\in [r,T]$ and using \eqref{eq:h1},
    \begin{multline}
    \label{eq:conv_phi_sk}
    d(\phi_{t,s_k}\circ\phi_{s_k,r}^{\pi_k}(a),\phi_{t,s}(y_s(a)))
    \\
    \leq
    d(\phi_{t,s_k}\circ\phi_{s_k,r}^{\pi_k}(a),\phi_{t,s_k}\circ y_s(a))+d(\phi_{t,s_k}\circ y_s(a),\phi_{t,s}\circ y_s(a))
    \\
    \leq (1+\delta_T)d(\phi_{s_k,r}^{\pi_k}(a),y_s(a))+\eta(\omega_{0,T})d(\phi_{s_k,r}^{\pi_k}(a),y_s(a))^\gamma
    \\
    +d(\phi_{t,s_k}\circ y_s(a),\phi_{t,s}\circ y_s(a)).
    \end{multline}
    Choosing $\Set{s_k}_{k\in\NN}$ so that $s_k$ decreases to $s$ and passing
    to the limit, we obtain with~\eqref{eq:conv_phi_sk} that $\phi_{t,s_k}\circ\phi_{s_k,r}^{\pi_k}(a)$ converges uniformly to $\phi_{t,s}\circ y_s(a)$. Thus, when $k\rightarrow +\infty$,   \eqref{eq:conv_pi_k} shows that $y$ is a D-solution. 
\end{proof}
\begin{lemma}
\label{lem:existence_solution}
Under Hypothesis~$1$, $G^{L_T}_\phi(r,a)$ is a non-empty compact
subset of the set of paths  $y\in\cC(\intSS,\uV)$ such that $y_r=a$
for any $r\in\TT$ and $a\in\uV$. It shows that \textup{H1} holds for $F:=G^{L_T}_\phi$.
\end{lemma}
\begin{proof}
It follows directly from Proposition~\ref{prop:sel:1} and Lemma~\ref{lem:sel:2} (with $K_T=1$ and $K=L_T$) that $G_\phi^K(r,a)$ is not empty. Now, if 
$\{y^k\}_k$ is a sequence in $G_\phi^K(r,a)$ then $\{y^k\}_k$ 
is equi-continuous with the same argument as in the proof of Proposition~\ref{prop:sel:1}. 
The subsequence of $\{y^k\}_k$ converges in
$G_\phi^K(r,a)$  because $a\in\uV\mapsto\phi_{t,s}(a)$ is continuous for any $(s,t)\in\rTT_2$.
\end{proof}

Let us denote $\tilde{G}^{L_T}_\phi(r,a)$, the set of  paths $y\in G^{L_T}_\phi(r,a)$ reparametrised
on $\itvcc{0,1}$ as $t\in \itvcc{0,1}\mapsto\tilde{y}_t:=y_{r+t(T-r)}$.

\begin{lemma}
\label{lem:G_closed}
Let us assume Hypothesis~\ref{hypo:varpi}. 
Let $r\geq 0$, for any closed subset $\cK\subset\cC(\itvcc{0,1},\uV)$, the set
$S'(r):=\Set{a\in\uV\given\tilde{G}^{L_T}_\phi(r,a)\bigcap\cK\ne\emptyset}$ is closed in $\uV$, in particular it is a Borel set in $\uV$. It shows that \textup{H2} holds for $F=G^{L_T}_\phi$.
\end{lemma}
\begin{proof}
Let $\{a_k\}_{k\in\NN}$ be a convergent sequence of $S'(r)$. For each $k\in\NN$,
we choose a path $\tilde{y}^k\in \tilde{G}_\phi^{L_T}(r,a_k)\bigcap \cK$ (which is not empty by definition). Then, for every $s,t\in \itvcc{0,1}$, $s\leq t$,
\begin{align}
\label{eq:equi-cont}
d(\tilde{y}^k_t,\tilde{y}^k_s)\leq d(\tilde{y}^k_t,\phi_{\tilde{t},\tilde{s}}(\tilde{y}^k_s))+d(\phi_{\tilde{t},\tilde{s}}(\tilde{y}^k_s),\tilde{y}^k_s)\leq [L_T\varpi(\omega_{\tilde{s},\tilde{t}})+\delta_{\tilde{t}-\tilde{s}}]\normsup{N},
\end{align}
where $\tilde{t}:=r+t(T-r)$ and $\tilde{s}:=r+s(T-r)$. Since $\tilde{t}-\tilde{s}$ goes to zero when $t-s\rightarrow 0$,
it follows that $\{t\in \itvcc{0,1}\mapsto \tilde{y}^k_t\}_{k\in\NN}$ is equi-continuous. 

The sequence $\{a_k\}_{k\in\NN}$ converges, so it is bounded by a constant
$A\geq 0$. Applying~\eqref{eq:equi-cont} between $s=0$ and $t$, we get
$\abs{\tilde{y}^k_t}\leq (L_T\varpi(\omega_{0,1})+\delta_{T})\normsup{N}+A$,
which proves that $t\in \itvcc{0,1}\mapsto y^k$ is uniformly bounded.

By Ascoli-Arzelà theorem, there is a convergent subsequence $\{\tilde{y}^{k_i}\}_{i\in\NN}$ in
$(\cC(\itvcc{0,1},\uV),\normsup{\cdot})$ to a path $y$. This path belongs to $\cK$ since $\cK$ is closed. Because $\phi_{t,s}$ is continuous, $\tilde{y}\in \tilde{G}^{L_T}_\phi(r,a)$.
Hence $S'(r)$ is closed and then Borel.
\end{proof}

The proof of the next lemma is an immediate consequence of the definition
of D-solutions.
\begin{lemma}
    \label{lem:sel:4}
    If $t\in[r,r+T]\mapsto y_t(r,a)$ belongs to $G^{L_T}_\phi(r,a)$,
    then for any $r'\geq 0$, its restriction $t\in[r+r',r+r'+T]\mapsto y_{t}(r,a)$ belongs to $G^{L_T}_\phi(r+r',y_{r+r'}(r,a))$. It shows that \textup{H3} holds for $F=G^{L_T}_\phi$.
\end{lemma}
\begin{lemma}
    \label{lem:sel:5}
    We assume that Hypothesis~\ref{hypo:varpi} hold.
   For $r'\geq 0$, if $y\in G^{L_T}_\phi(r,a)$ and $z\in G^{L_T}_\phi(r+r',y_{r+r'}(r,a))$, 
   then $y\bowtie_{r+r'} z\in G^{L_T}_\phi(r,a)$. It shows that \textup{H4} holds for $F:=G^{L_T}_\phi$ 
\end{lemma}
\begin{proof}
    Let us write $x:=y\bowtie_s z$ where $s:=r+r'$ and $U_{\tau,t}:=d(x_t,\phi_{t,\tau}(x_\tau))$ for $\tau\leq t$. On the one hand, for any $r\leq \tau\leq s\leq t$
   with \eqref{eq:def_davie}, \eqref{eq:h2} and \eqref{eq:h3},
\begin{align}
\nonumber
U_{\tau,t}&\leq d(x_t,\phi_{t,s}(x_s))
	+d(\phi_{t,s}(x_s),\phi_{t,s}\circ\phi_{s,\tau}(x_\tau))
	+d(\phi_{t,s}\circ\phi_{s,\tau}(x_\tau),\phi_{t,\tau}(x_\tau))\\
	\label{eq:concatenation}
&\leq \normsup{N}(2+\delta_T)L_T\varpi(\omega_{\tau,t}).
\end{align}    
On the other hand, for $s\leq \tau\leq t$ or $\tau \leq t\leq s$ with \eqref{eq:def_davie}
\begin{align}
\label{eq:concatenation_2}
U_{\tau,t}\leq \normsup{N}L_T\varpi(\omega_{\tau,t}).
\end{align}
Thus, combining \eqref{eq:concatenation} and \eqref{eq:concatenation_2}, 
for any $r\leq\tau\leq t\leq T$,
\begin{align*}
U_{\tau,t}\leq \normsup{N}(2+\delta_T)L_T\varpi(\omega_{\tau,t}).
\end{align*}
 Besides, for any $r\leq \tau \leq u\leq t\leq T$ with \eqref{eq:h2} and \eqref{eq:h3},
 \begin{align}
 \label{eq:U_tau_t}
 U_{\tau,t}\leq U_{u,t}+(1+\delta_T)U_{\tau,u}+\normsup{N}\varpi(\omega_{\tau,t}).
 \end{align}

Let $\lambda\in (0,1)$ such that $\varpi^\lambda$ satisfies~\eqref{eq:h4} with $\varkappa_\lambda:=2^{1-\lambda}\varkappa^\lambda<1$.
Let $T(\lambda)>0$ be a real number such that $\varkappa_\lambda (1+\delta_{T(\lambda)})<1$.
For any, two successive  points $\tau,t$ of a subdivision~$\pi$,
\begin{align}
U_{\tau,t}\leq D_{L_T}(\pi,\lambda)\varpi^\lambda(\omega_{\tau,t}),\label{eq:induction_1}
\end{align}
where $D_{L_T}(\pi,\lambda):=\normsup{N}(2+\delta_T)L_T\sup_{\substack{\tau,t\text{~successive points of }\pi}}\varpi^{1-\lambda}(\omega_{\tau,t})$.

Let us show by induction over the distance $m$ between points $\tau$ and $t$ in $\pi\cap [0,T(\lambda)]$ that
\begin{align}
U_{\tau,t}\leq A_{L_T}(\pi,\lambda)\varpi^\lambda(\omega_{\tau,t}),\label{eq:h_induction}
\end{align}
where
\begin{equation*}
A_{L_T}(\pi,\lambda):= \frac{D_{L_T}(\pi,\lambda)(1+\delta_{T(\lambda)})+2\normsup{N}\varpi^\lambda(\omega_{0,T(\lambda)})}{1-\varkappa_\lambda (1+\delta_{T(\lambda)})}.
\end{equation*}
When $m=0$, $U_{\tau,\tau}=0$ so that \eqref{eq:h_induction} holds.
For $m=1$, $\tau$ and $t$ are successive points
then \eqref{eq:h_induction} holds with \eqref{eq:induction_1}.
Now, we assume that \eqref{eq:h_induction} holds for any 
two points at distance $m$. Let $\tau$ and $t$ be two points at distance $m+1$ in $\pi\cap[0,T(\lambda)]$.
Since $\omega$ is super-additive, one may choose two successive points $s$ and $s'$ 
in $\pi$ with $\tau<s<s'<t$ such that 
$\omega_{\tau,s}\leq \omega_{\tau,t}/2$ and $\omega_{s',t}\leq \omega_{\tau,t}/2$, as 
in the proof of Theorem~\ref{thm:1}.
Then, by applying \eqref{eq:U_tau_t} between $(\tau,s,s')$
and $(s,s',t)$ we obtain,
\begin{align*}
U_{\tau,t}&
    \leq U_{s,t}+(1+\delta_{T(\lambda)})U_{\tau,s}+\normsup{N}\varpi^{1-\lambda}(\omega_{0,T(\lambda)})\varpi^\lambda(\omega_{\tau,t})\\
&\leq U_{s',t}+(1+\delta_{T(\lambda)})U_{s,s'}+(1+\delta_{T(\lambda)})U_{\tau,s}+2\normsup{N}\varpi^{1-\lambda}(\omega_{0,T(\lambda)})\varpi^\lambda(\omega_{\tau,t})\\
&\leq [A_{L_T}(\pi,\lambda)\varkappa_\lambda(1+\delta_{T(\lambda)})+(1+\delta_{T(\lambda)})D_{L_T}(\pi,\lambda)+2\normsup{N}\varpi^{1-\lambda}(\omega_{0,T(\lambda)})]\varpi^\lambda(\omega_{\tau,t})\\
&\leq A_{L_T}(\pi,\lambda)\varpi^\lambda(\omega_{\tau,t}),
\end{align*}
with our choice of $A_{L_T}(\pi,\lambda)$. 
This concludes the induction, so \eqref{eq:h_induction} holds for any
$\tau,t\in\pi\cap [0,T(\lambda)]^2$.

Clearly, $D_{L_T}(\pi,\lambda)\rightarrow 0$ when the mesh of $\pi$ goes to
zero. Then, $A_{L_T}(\pi,\lambda)\rightarrow
A(\lambda):=2\normsup{N}\varpi^{\lambda}(\omega_{0,T(\lambda)})/(1-\varkappa_\lambda(1+\delta_{T(\lambda)}))$ when the mesh of $\pi$ goes to zero.

By continuity of $(\tau,t)\mapsto U_{\tau,t}$, considering finer and finer partitions leads to 
$U_{\tau,t}\leq A(\lambda)\varpi^\lambda(\omega_{\tau,t})$ for
any $r\leq \tau\leq t\leq T(\lambda)$.

Finally, choosing $T(\lambda)$ so that $T(\lambda)$ increases to $T$ defined in Hypothesis~\ref{hypo:varpi} when $\lambda$ goes to $1$, we conclude that for any $ r\leq\tau\leq t\leq T$,
\begin{align*}
U_{\tau,t}\leq \normsup{N}{L_T}\varpi(\omega_{\tau,t}),
\end{align*}
where $L_T$ is defined in Hypothesis~\ref{hypo:varpi}.
This proves that $z\in G_\phi^{L_T}(r,a)$.
\end{proof}

\begin{proof}[Proof of Theorem~\ref{thm:nls:1}]
    Lemma~\ref{lem:sel:T}-\ref{lem:sel:5} prove that conditions \textup{H0-H4} of Definition~\ref{def:local_funnel} hold for $F=G_\phi^{L_T}$. This means that $G^{L_T}(r,a)$ is a family of abstract local integral funnels.
    We apply Theorem~1 in \cite{cardona_semiflow2}. For any $(r,a)\in\TT\times\uV$, 
    there exists a measurable map $a\mapsto (t\mapsto \psi_{t,r}(a))$
    with respect to the Borel
    subsets of $\cC^0(\TT,\uV)$ with the property
    that $\psi_{r,r}(a)=a$ and $\psi_{t,s}\circ\psi_{s,r}(a)=\psi_{t,r}(a)$, $t\geq r$. 
\end{proof}



\section{Lipschitz flows}
\label{sec:lipschitz}
A Lipschitz almost flow which has the flow property is said to be
a Lipschitz flow. We recast the definition.

\begin{definition}[Lipschitz flow]
    \label{def:lipschitz:flow}
    A flow $\psi\in\cF(\uV)$ is said to be a \emph{Lipschitz flow}
    if for any $(s,t)\in\rTT_2$, $\psi_{t,s}$ is Lipschitz in 
    space with $\normlip{\psi_{t,s}}\leq 1+\delta_T$.
\end{definition}

In this section, we consider galaxies that contain a Lipschitz flow. 

We prove that such a Lipschitz flow $\psi$ is the only possible flow in
the galaxy (Theorem~\ref{thm:8}), and that the iterated almost flow $\phi^\pi$ of any almost flow $\phi$
converges to $\psi$ (Theorem~\ref{thm:7}). We also characterize the rate of convergence (Theorem~\ref{thm:rate_conv}).

Let us choose $\lambda\in(0,1)$ such that $\varpi^\lambda$ satisfies
the same properties as $\varpi$ up to changing $\varkappa$ to 
$\varkappa_\lambda\eqdef 2^{1-\lambda}\varkappa^\lambda$, provided
$\varkappa_\lambda<1$. This is possible as soon as $\lambda>1/(1-\log_2(\varkappa))$ with \eqref{eq:h4}.

Clearly, if for $\psi,\chi\in\cF(\uV)$, $\Delta_{N,\varpi}(\psi,\chi)<+\infty$, 
then 
\begin{equation}
    \label{eq:42}
\Delta_{N,\varpi^\lambda}(\psi,\chi)\leq \Delta_{N,\varpi}(\psi,\chi)
\varpi^{1-\lambda}(\omega_{0,T})<+\infty,
\end{equation}
where $\varpi$ is defined by \eqref{eq:h4}.
Hence, the galaxies remain the same when $\varpi$ is changed to $\varpi^\lambda$.
We define
\begin{align}
    \label{eq:theta}
    \Theta(\pi)\eqdef\sup_{d_\pi(s,s')=1}\varpi^{1-\lambda}(\omega_{s,s'}).
\end{align}

\begin{theorem}
    \label{thm:7}
    Let $\phi$ be an almost flow such that $\normlip{\phi^\pi}\leq 1+\delta_T$
    whatever the partition $\pi$, we say that $\phi$ satisfies the uniform Lipschitz (UL) condition. Then there exists a Lipschitz flow $\zeta\in\cF(\uV)$
    with $\normlip{\zeta_{s,t}}\leq 1+\delta_T$ 
    such that $\{\phi^\pi\}$ converges to~$\zeta$
    as $\abs{\pi}\rightarrow 0$.
\end{theorem}
\begin{theorem}
\label{thm:rate_conv}
    Let $\phi$ be an almost flow and $\psi$ be a
    Lipschitz flow with $\psi\sim\phi$. Then 
    there exists a constant $K$ that depends only on $\lambda$, $\Delta_{N,\varpi}(\phi,\psi)$,
    $\varkappa$ and $T$ (assumed to be small enough) so that 
    \begin{equation*}
	\Delta_{N,\varpi^\lambda}(\psi,\phi^\pi)
	\leq K\Theta(\pi).
    \end{equation*}
    In particular, $\Set{\phi^\pi}_\pi$ converges to $\psi$
    as $\abs{\pi}\rightarrow 0$.
\end{theorem}

\begin{remark}
    In \cite{brault2,brault3}, we develop the notion of \emph{stable almost flow}
    around a necessary condition for an almost flow to be associated to a Lipschitz flow.
    Under such a condition, a stronger rate of convergence may be achieved 
    by taking $\Theta(\pi):=\sup_{d_\pi(s,s')=1}\varpi(\omega_{s,s'})/\omega_{s,s'}$
    \cite{brault3}.
\end{remark}

\begin{theorem}[Uniqueness of Lipschitz flows]
    \label{thm:8}
    If $\psi$ is a Lipschitz flow and $\chi$ is 
    a flow (not necessarily Lipschitz \textit{a priori}) 
    in the same galaxy as $\psi$, that is $\chi\sim\psi$, then $\chi=\psi$.
\end{theorem}


\begin{hypothesis}
    \label{hyp:uniqueness:pi}
Let us fix a partition $\pi$.
We consider $\psi$ and $\chi$ in $\cF(\uV)$ such that $\psi\sim\chi$ and 
for any $(r,s,t)\in\pi^3_+$, 
\begin{align}
    \label{eq:hyp:uniq:pi:1}
    \normlip{\psi_{t,s}}&\leq 1+\delta_T, \\
    \label{eq:hyp:uniq:pi:2}
    N(\chi_{t,s}(a))&\leq (1+\delta_T)N(a),\ \forall a\in\uV,\\
    \label{eq:hyp:uniq:pi:3}
    \Delta_{N}(\psi_{t,s}\circ \psi_{s,r},\psi_{t,r})&\leq \beta_\psi \varpi(\omega_{r,t}) 
    \text{ and }
    \Delta_{N}(\chi_{t,s}\circ \chi_{s,r},\chi_{t,r})\leq \beta_\chi \varpi(\omega_{r,t}),
\end{align}
for some constant $\beta_\chi,\beta_\phi\geq 0$.
\end{hypothesis}

\begin{remark}
    \label{rem:2}
In Hypothesis~\ref{hyp:uniqueness:pi}, the role of $\psi$ and $\chi$ are not exchangeable:
$\psi$ is assumed to be Lipschitz, there is no such requirement on $\chi$.
The reason of this dissymmetry lies in \eqref{eq:43}.
\end{remark}

\begin{remark}
    \label{rem:1}
If $\psi$ is a Lipschitz almost flow and $\chi$ is an almost flow,
then $(\psi,\chi)$ satisfies Hypothesis~\ref{hyp:uniqueness:pi} for any partition $\pi$. The condition \eqref{eq:hyp:uniq:pi:2} is a particular case of \eqref{eq:thm1_KT}.
\end{remark}

We choose   $\lambda$ and $T$ so that
\begin{equation*}
    \frac{1}{1-\log_2(\varkappa)}<\lambda<1
    \text{ and }
    3\delta_T+\delta_T^3<2\frac{1-\varkappa_\lambda}{\varkappa_\lambda}.
\end{equation*}
We define (recall that $\Theta(\pi)$ is given by \eqref{eq:theta}),
\begin{align*}
    \rho_T&:=\varpi(\omega_{0,T})^{1-\lambda},\\
    \gamma(\pi)&:=\sup_{d_\pi(s,s')=1}\frac{\Delta_{N}(\psi,\chi)}{\varpi^\lambda(\omega_{s,s'})}\leq \Delta_{N,\varpi^\lambda}(\psi,\chi)\Theta(\pi),\\
    \beta(\pi)&:=(2+3\delta_T+\delta_T^2)(\beta_\psi+\beta_\chi)\rho_T
    +(1+\delta_T)^2\gamma(\pi)\geq \gamma(\pi),\\
    \text{ and }
    L(\pi)&:=\frac{2\beta(\pi)}{2-\varkappa_\lambda(2+3\delta_T+\delta_T^2)}\geq \gamma(\pi).
\end{align*}
Here, $\Theta(\pi)$ and thus $\gamma(\pi)$ converge to zero when the mesh of
$\pi$ tends to zero.
\begin{lemma}
    \label{thm:2}
    Let $\phi,\chi\in\cF(\uV)$ and $\pi$ be satisfying Hypothesis~\ref{hyp:uniqueness:pi}.
    With the above choice of $\lambda$ and $T$, it holds that 
    \begin{equation}
	\label{eq:41}
	d(\phi_{t,r}(a),\chi_{t,r}(a))
	\leq L(\pi\cup\Set{t,r})N(a)\varpi^\lambda (\omega_{r,t}),\ \forall (r,t)\in\TT_+^2. 
    \end{equation}
\end{lemma}

\begin{proof}
We set $F_{r,t}\eqdef\Delta_{N}(\psi_{t,r},\chi_{t,r})$,
where $\Delta_N$ is defined in Notation~\ref{not:norms}.
From Definition~\ref{def_galaxy}, $F_{r,t}\leq \Delta_{N,\varpi}(\psi,\chi)\varpi(\omega_{r,t})<+\infty$ 
since $\psi\sim\chi$.

In particular, for $(r,s,t)\in\pi^3_+$, with \eqref{eq:hyp:uniq:pi:2} in Hypothesis~\ref{hyp:uniqueness:pi},
\begin{equation}
    \label{eq:36}
d(\psi_{t,s}\circ \chi_{s,r}(a),\chi_{t,s}\circ\chi_{s,r}(a))
\leq F_{s,t} N(\chi_{s,r}(a))\leq (1+\delta_T)N(a) F_{s,t}.
\end{equation}

For any $(r,s,t)\in\pi^+_3$, the fact that $\phi$, $\chi$ are almost flow combined with \eqref{eq:hyp:uniq:pi:1}-\eqref{eq:hyp:uniq:pi:3} and \eqref{eq:36} imply that 
for any $a\in\uV$, 
\begin{multline}
    \label{eq:43}
d(\psi_{t,r}(a),\chi_{t,r}(a))\\
\leq d(\psi_{t,s}\circ \psi_{s,r}(a),\psi_{t,s}\circ\chi_{s,r}(a))
+d(\psi_{t,s}\circ \chi_{s,r}(a),\chi_{t,s}\circ\chi_{s,r}(a))
+(\beta_\phi+\beta_\chi) N(a)\varpi(\omega_{r,t})\\
\leq (1+\delta_T)N(a)F_{r,s}+(1+\delta_T)N(a)F_{s,t}+(\beta_\chi+\beta_\phi) N(a)\varpi(\omega_{r,t}).
\end{multline}
Thus, dividing by $N(a)$,
\begin{equation}
    \label{eq:37}
    F_{r,t}\leq (1+\delta_T)(F_{r,s}+F_{s,t})+(\beta_\chi+\beta_\psi)\rho_T\varpi^\lambda(\omega_{r,t}).
\end{equation}

We proceed by induction. Our hypothesis is that 
\begin{align}
\label{eq:induction_hyp}
F_{r,t}\leq L(\pi)\varpi^\lambda(\omega_{r,t}),
\forall(r,t)\in\pi^+_2,~\text{at distance at most } m.
\end{align}

When $m=0$, $F_{r,r}=0$ since $\psi_{r,r}(a)=\chi_{r,r}(a)=a$ for any $a\in\uV$. Thus \eqref{eq:induction:1} is true for $m=0$.
When $m=1$, $r$ and $t$ are successive points. From the very definition
of $\gamma(\pi)$, 
\begin{equation}
    \label{eq:40}
    F_{r,t}\leq \gamma(\pi)\varpi^\lambda(\omega_{r,t}).
\end{equation}
The induction hypothesis~\eqref{eq:induction:1} is true for $m=1$ since $L(\pi)\geq \gamma(\pi)$.

Assume that the induction hypothesis is true at some level $m\geq 1$. 
Let $(r,s,t)\in\pi^+_3$ with $r<s<t$ and $d_\pi(r,t)=m+1$. 
Let $s'$ be such that $s$ and $s'$ are successive points in $\pi$
(possibly, $s'=t$). Clearly, $d_\pi(r,s)\leq m$  and $d_\pi(s',t)\leq m$.
Using \eqref{eq:37} to decompose $F_{s,t}$ using $s'$ and using \eqref{eq:40}, 
\begin{multline*}
    F_{r,t}\leq 
    (1+\delta_T)F_{r,s}
    +(1+\delta_T)^2F_{s',t}
    +(1+\delta_T)^2\gamma(\pi)\varpi^\lambda(\omega_{s,s'})\\
+(2+3\delta_T+\delta_T^2)(\beta_\psi+\beta_\chi)\rho_T\varpi^\lambda(\omega_{r,t})
\\
\leq (1+\delta_T)F_{r,s}+(1+\delta_T)^2F_{s',t}+\beta(\pi)\varpi^\lambda(\omega_{r,t}).
\end{multline*}
With the induction hypothesis, since $r$ and $s$ (resp. $s'$ and $t$) are at
distance at most $m$, 
\begin{equation*}
    F_{r,t}\leq L(\pi)(1+\delta_T)\varpi^\lambda(\omega_{r,s})+L(\pi)(1+\delta_T)^2\varpi^\lambda(\omega_{s',t})
    +\beta(\pi)\varpi^\lambda(\omega_{r,t}).
\end{equation*}
Choosing $s$ and $s'$ to satisfy \eqref{eq:38}, our choice of $L(\pi)$
and \eqref{eq:h4} imply that 
\begin{equation*}
    F_{r,t}\leq \left(L(\pi)\frac{2+3\delta_T+\delta_T^2}{2}\varkappa_\lambda+\beta(\pi)\right)\varpi^\lambda(\omega_{r,t})
    \leq L(\pi)\varpi^\lambda(\omega_{r,t}).
\end{equation*}
The induction hypothesis~\eqref{eq:induction_hyp} is then true at level $m+1$, and then whatever the distance
between the points of the partition.

Finally, \eqref{eq:41} is obtained by replacing $\pi$ by $\pi\cup\Set{r,t}$.
\end{proof}
\begin{proof}[Proof Theorem~\ref{thm:7}] Let $\sigma$ and $\pi$ be two partitions with $\pi\subset\sigma$. 
    We set $\psi\eqdef\phi^\sigma$ and $\chi\eqdef\phi^\pi$. 

    With Theorem~\ref{thm:1},
    \begin{equation*}
	\Delta_{N,\varpi}(\phi^\sigma,\phi^\pi)
	\leq 
	\Delta_{N,\varpi}(\phi^\sigma,\phi)
	+
	\Delta_{N,\varpi}(\phi^\pi,\phi)\leq 2L_T.
    \end{equation*}
    With \eqref{eq:42},
    $\Delta_{N,\varpi^\lambda}(\psi,\chi)\leq 2L_T\rho_T$,
    so that $\Set{\Delta_{N,\varpi^\lambda}(\psi,\chi)}_{\pi,\sigma}$ is bounded.

    Again with Theorem~\ref{thm:1}, $(\psi,\chi)$ satisfies Hypothesis~\ref{hyp:uniqueness:pi} for the subdivision $\pi$
    (up to changing $\delta_T$) with $\beta_\psi=\beta_\chi=0$.

    Hence, $L(\pi)=C\gamma(\pi)$ where 
    \begin{equation}
	\label{eq:cst:L}
	C:=\frac{2(1+\delta_T)^2}{(2-\varkappa_\lambda(2+3\delta_T+\delta_T^2)}.
    \end{equation}

    We may then rewrite \eqref{eq:41} as 
    \begin{align}
    \label{eq:cauchy}
	d(\phi^\sigma_{t,r}(a),\phi^\pi_{t,r}(a))\leq C\gamma(\pi\cup\Set{r,t})N(a)\varpi^\lambda(\omega_{r,t}).
    \end{align}
    Since $\gamma(\pi)$ decreases to $0$ as $\abs{\pi}$ decreases to $0$ 
    and $\abs{\pi\cup\Set{r,t}}\leq \abs{\pi}$, 
    it is easily shown that $\Set{\phi^\pi_{t,s}}_{\pi}$ forms a Cauchy net
    with respect to the nested partitions. Then, it does converges
    to a limit $\zeta_{s,t}(a)$. By Theorem~\ref{thm:2} and the continuity of $N$, $N(\zeta_{s,r})\leq K_T N(a)$. From the UL condition, $a\mapsto \zeta_{t,s}(a)$
    is Lipschitz continuous with $\normlip{\zeta_{t,s}}\leq 1+\delta_T$.

Moreover $\zeta$ does not depend on the subdivision $\pi$. Indeed, if
$\tilde{\pi}$ is another subdivision, we obtain with \eqref{eq:cauchy},
that $\{\phi^{\tilde{\pi}}\}_{\tilde{\pi}}$ converges to $\zeta$
when $|\tilde{\pi}|\rightarrow 0$.

Finally, if if $\Set{\pi_k}_{k\geq 0}$ is a family of nested partitions, and $(r,s,t)\in\rTT_3$,
\begin{align*}
\phi_{t,r}^{\pi_k\cup\{s\}}=\phi_{t,s}^{\pi_k}\circ\phi_{s,r}^{\pi_k}.
\end{align*}
Because $|\pi_k\cup \{s\}|\leq |\pi_k|$ and \eqref{eq:cauchy}, $\phi^{\pi_k\cup\{s\}}$ converges to $\zeta$ when $k\rightarrow +\infty$. Moreover, for any $a\in\uV$,
\begin{align*}
d(\zeta_{t,s}\circ\zeta_{s,r}(a),\phi_{t,s}^{\pi_k}\circ\phi_{s,r}^{\pi_k}(a))&\leq
d(\zeta_{t,s}\circ\zeta_{s,r}(a),\phi_{t,s}^{\pi_k}\circ\zeta_{s,r}(a))
+d(\phi_{t,s}^{\pi_k}\circ\zeta_{s,r}(a),\phi_{t,s}^{\pi_k}\circ\phi_{s,r}^{\pi_k}(a))\\
&\leq C\gamma(\pi_k\cup\Set{r,t})N(\zeta_{s,r}(a))\varpi^\lambda(\omega_{r,t})+(1+\delta_T)d(\zeta_{s,r}(a),\phi_{s,r}^{\pi_k}(a))\\
&\leq C\gamma(\pi_k\cup\Set{r,t})(1+K_T) N(a)\varpi(\omega_{r,t}),
\end{align*}
because $N(\zeta_{s,r}(a))\leq K_TN(a)$. So, $\{\phi_{t,s}^{\pi_k}\circ\phi_{s,r}^{\pi_k}\}_{\pi_k}$ converges uniformly to
$\zeta_{t,s}\circ\zeta_{s,r}$ when $m\rightarrow +\infty$.
Then, the flow property $\zeta_{t,s}\circ\zeta_{s,r}=\zeta_{t,r}$ holds.
\end{proof}
\begin{proof}[Proof Theorem~\ref{thm:rate_conv}]
    For a partition $\pi$, the pair $(\psi,\phi^\pi)$ satisfies
    Hypothesis~\ref{hyp:uniqueness:pi} for the subdivision $\pi$ with $\beta_\psi=\beta_\chi=0$. 
    As in the proof of Theorem~\ref{thm:7}
    (we have assumed for convenience that $\Delta_{N,\varpi}(\phi,\psi)\leq L_T$), 
    \begin{equation*}
	\Delta_{N,\varpi^\lambda}(\psi,\phi^\pi)\leq C\gamma(\pi)
	\leq 2CL_T\rho_T\Theta(\pi)
    \end{equation*}
    for $C$ given by \eqref{eq:cst:L}.
    This proves the result.
\end{proof}
\begin{proof}[Proof Theorem~\ref{thm:8}]
    For any partition $\pi$, $\psi$ and $\chi$ satisfy Hypothesis~\ref{hyp:uniqueness:pi}
    with $\beta_\psi=\beta_\chi=0$. Thus, 
    \begin{equation*}
	\Delta_{N,\varpi^\lambda}(\psi,\chi)\leq C\gamma(\pi)
    \end{equation*}
    with $C$ given by \eqref{eq:cst:L}. As $\gamma(\pi)$ decreases to $0$ 
    when $\abs{\pi}$ decreases to $0$, we obtain that $\psi=\chi$. 
\end{proof}
\begin{corollary}
    \label{cor:9}
    Let $\psi$ and $\chi$ be two almost flows with $\psi\sim\chi$
    and $\psi$ be Lipschitz.
    Then for $T$ small enough (in function of some $\lambda<1$, $\varkappa$ and $\delta$)
    \begin{equation*}
	\Delta_{N,\varpi}(\psi,\chi)
	\leq 
	    \frac{2(2+3\delta_T+\delta_T^2)(\beta_\psi+\beta_\chi)}
	    {2-\varkappa_\lambda(2+3\delta_T+\delta_T^2)}.
    \end{equation*}
\end{corollary}
\begin{proof}
    With Remark~\ref{rem:1}, $(\psi,\chi)$ satisfies Hypothesis~\ref{hyp:uniqueness:pi}.
    Letting the mesh of the partition decreasing to $0$
    as the in proof of Theorem~\ref{thm:8}, and then letting
    $\lambda$ increasing to $1$ leads to the result.
\end{proof}

\section{Perturbations}
\label{sec:perturbation}
In this section, we consider the construction of an almost flow by perturbations
of existing ones. We assume that $\uV$ is a Banach space.

Let $\phi\in\cF(\uV)$ be an almost flow 
with respect to a function $N_\gamma$ such that $N_\gamma(a)\geq N_\gamma(0)\geq 1$.

\begin{notation}
    \label{not:phirst}
    For $\phi\in\cF(\uV)$ when $\uV$ is a Banach space, we write
   \begin{equation*}
       \phi_{t,s,r}(a):=\phi_{t,s}(\phi_{s,r}(a))-\phi_{t,r}(a).
   \end{equation*}
\end{notation}

\begin{definition}
    \label{def:perturb}
    Let $\epsilon\in\cF(\uV)$ such that 
     for any $(s,t)\in\rTT_2$, $a,b\in\uV$,
    \begin{gather}
	\label{eq:epsilon:1}
	\epsilon_{t,t}\equiv 0,\\
	\label{eq:epsilon:2}
	\abs{\epsilon_{t,s}(a)}\leq \lambda N_\gamma(a)\varpi(\omega_{s,t}),\\
	\label{eq:epsilon:3}
	\abs{\epsilon_{t,s}(b)-\epsilon_{t,s}(a)}\leq \eta(\omega_{s,t})\abs{b-a}^\gamma
    \end{gather}
    for some $\lambda\geq 0$. We say that $\epsilon$ is a \emph{perturbation}.
\end{definition}

\begin{proposition} 
    \label{prop:pert:1}
    If $\phi\in\cF(\uV)$ is an almost flow and $\epsilon\in\cF(\uV)$
    is a perturbation, then $\psi\eqdef\phi+\epsilon$ is an almost flow. 
    Besides, $\psi\sim\phi$.
\end{proposition}
\begin{proof}
Let $(r,s,t)\in\rTT_3$ and $a,b\in\uV$.
   From \eqref{eq:epsilon:1}, \eqref{eq:h0} is satisfied. 
   With $\delta'_T\eqdef\delta_T+\lambda\varpi(\omega_{0,T})$, 
   \eqref{eq:h1} is also true.    In addition, with \eqref{eq:epsilon:3}, 
   \begin{align*}
   \abs{\psi_{t,s}(b)-\psi_{t,s}(a)}\leq (1+\delta_T)\abs{b-a}+2\eta(\omega_{s,t})\abs{b-a}^\gamma. 
   \end{align*}
   Thus, $\psi$ satisfies \eqref{eq:h2}.

To show~\eqref{eq:h3}, we write
\begin{multline}
\psi_{t,s,r}(a)
=\phi_{t,s}\circ\psi_{s,r}(a)+\epsilon_{t,s}\circ\psi_{s,r}(a)-\phi_{t,r}(a)-\epsilon_{t,s}(a)\\
=\underbrace{\phi_{t,s,r}(a)}_{\run_{r,s,t}}+\underbrace{\phi_{t,s}\circ(\phi_{s,r}+\epsilon_{s,r})(a)-\phi_{t,s}\circ\phi_{s,r}(a)}_{\rdeux_{r,s,t}}\\ 
+\underbrace{\epsilon_{t,s}\circ(\phi_{s,r}(a)+\epsilon_{s,r}(a))-\epsilon_{t,s}\circ\phi_{s,r}(a)}_{\rtrois_{r,s,t}}
+\underbrace{\epsilon_{t,s}\circ\phi_{s,r}(a)-\epsilon_{t,s}(a)}_{\rquatre_{r,s,t}}.
\label{eq:perturbation_decomposition}
\end{multline}
We control the first term with \eqref{eq:h3}, $\abs{\run_{r,s,t}}\leq N_\gamma(a)\varpi(\omega_{r,t})$. For the second one, we use \eqref{eq:h_eta}, \eqref{eq:h2} and \eqref{eq:epsilon:2},
\begin{multline*}
\abs{\rdeux_{r,s,t}}\leq (1+\delta_T)\abs{\epsilon_{s,r}(a)}+\eta(\omega_{s,r})\abs{\epsilon_{s,r}}^{\gamma}\\
\leq (1+\delta_T)\lambda N_\gamma(a)\varpi(\omega_{r,s})+\eta(\omega_{s,t})\lambda^\gamma N_\gamma^\gamma(a)
\varpi(\omega_{r,s})^\gamma\\
\leq [1+(1+\lambda^\gamma)\delta_T]N_\gamma(a)\varpi(\omega_{r,t}),
\end{multline*}
because $N_\gamma(a)\geq 1$ implies that $N_\gamma^\gamma(a)\leq N_\gamma(a)$ for $\gamma\in (0,1)$.

With \eqref{eq:epsilon:2} and \eqref{eq:epsilon:3}, we obtain for the third term,
\begin{equation*}
\abs{\rtrois_{r,s,t}}\leq \eta(\omega_{s,t})\abs{\epsilon_{s,r}(a)}^\gamma
\leq \lambda^\gamma N_\gamma(a)^\gamma\eta(\omega_{s,t})\varpi(\omega_{r,s})^\gamma
\leq \lambda^\gamma N_\gamma(a)\delta_T\varpi(\omega_{s,t}),
\end{equation*}
where the last inequality comes from \eqref{eq:h_eta}.
And for the last term, we use \eqref{eq:thm1_KT} and~\eqref{eq:epsilon:2},
\begin{multline*}
\abs{\rquatre_{r,s,t}}\leq \abs{\epsilon_{t,s}\circ\phi_{s,r}(a)}+\abs{\epsilon_{t,s}(a)}
\leq ( N_\gamma(\phi_{s,r})+N_{\gamma}(a))\lambda\varpi(\omega_{s,t})\\
\leq (K_T+1)N_\gamma(a)\lambda\varpi(\omega_{s,t}).
\end{multline*}
Thus, combining estimations for each four terms of \eqref{eq:perturbation_decomposition}, we obtain \eqref{eq:h3} which proves that $\psi$ is an almost flow.

Besides,
\begin{equation*}
\abs{\psi_{t,s}(a)-\phi_{t,s}(a)} = \abs{\epsilon_{t,s}(a)}
\leq \lambda N_\gamma(a)\varpi(\omega_{s,t}),
\end{equation*}
which proves that $\psi\sim \phi$ and concludes the proof.
\end{proof}

\section{Applications}
\label{sec:applications}
On this section, we show that our framework covers former different sewing lemmas.
\subsection{The additive sewing lemma}

The additive sewing lemma is the key to construct
the Young integral \cite{young36a} and the rough integral \cite{lyons98a,gub04}.

We consider that $\uV$ is a Banach space with a norm 
$\eucnorm{\cdot}$. The distance $d$ is $d(a,b):=\eucnorm{b-a}$.

\begin{definition}[Almost additive functional] A family $\Set{\alpha_{s,t}}_{(s,t)\in\rTT_2}$ 
    is an \emph{almost additive functional} if 
    \begin{equation*}
	\alpha_{r,s,t}\eqdef \alpha_{r,s}+\alpha_{s,t}-\alpha_{r,t}\text{ satisfies }
	\eucnorm{\alpha_{r,s,t}}\leq \varpi(\omega_{r,t}),\ \forall (r,s,t)\in\rTT_3.
    \end{equation*}
    It is an \emph{additive functional} if $\alpha_{r,s,t}=0$ for any $(r,s,t)\in\rTT_3$.
\end{definition}

\begin{proposition}[The additive sewing lemma \cite{lyons02b,feyel}] If $\Set{\alpha_{s,t}}_{(s,t)\in\rTT_2}$ 
    is an almost additive functional with $\eucnorm{\alpha_{s,t}}\leq \delta_T$, 
    there exists an additive functional 
    $\Set{\gamma_{s,t}}_{(s,t)\in\rTT_2}$ which is unique in the sense   that for any constant $C\geq 0$ and any additive functional 
    $\Set{\beta_{s,t}}_{(s,t)\in\rTT_2}$, 
    $\eucnorm{\beta_{s,t}-\alpha_{s,t}}\leq C\varpi(\omega_{s,t})$ implies that $\beta=\gamma$.
\end{proposition}

\begin{proof}
Clearly, $\phi_{t,s}(a)=a+\alpha_{s,t}$ is an almost flow which satisfies the UL condition. Hence 
the result.
\end{proof}

\subsection{The multiplicative sewing lemma}

Here we recover the results of \cite{lyons98a,feyel,coutin-lejay1}.
We consider now that the metric space $\uV$ has a monoid structure:
there exists a product $ab\in\uV$ of two elements $a,b\in\uV$.
We also assume that there exists a Lipschitz function $N:\uV\to[1,+\infty)$
such that 
\begin{equation*}
d(ac,bc)\leq N(c)d(a,b)\text{ and }d(ca,cb)\leq N(c)d(a,b)
\text{ for all }a,b,c\in\uV.
\end{equation*}

\begin{definition} A family $\Set{\alpha_{s,t}}_{(s,t)\in\rTT_2}$ is said to be an \emph{almost multiplicative
	functional} if
    \begin{equation*}
    d(\alpha_{r,s}\alpha_{s,t},\alpha_{r,t})\leq \varpi(\omega_{r,t}),\ \forall (r,s,t)\in\TT_+^3.
    \end{equation*}
    It is a \emph{multiplicative functional} if $\alpha_{r,s}\alpha_{s,t}=\alpha_{r,t}$.
\end{definition}

\begin{proposition}[The multiplicative sewing lemma \cite{feyel}] 
    If $\Set{\alpha_{s,t}}_{(s,t)\in\TT_+^2}$
is an \emph{almost multiplicative functional} then there exists a unique 
multiplicative functional $\Set{\gamma_{s,t}}_{(s,t)\in\TT_+^2}$ such that
any other multiplicative functional $\Set{\gamma_{s,t}}_{(s,t)\in\TT_+^2}$ such
that $d(\beta_{s,t},\alpha_{s,t})\leq C\varpi(\omega_{s,t})$ for any $(s,t)\in\TT_+^2$ satisfies $\beta=\gamma$.
\end{proposition}

\begin{proof}
    For this, it is sufficient to consider $\phi_{t,s}(a)=a\alpha_{s,t}$ which is an
    almost flow which satisfies the UL condition.
\end{proof}

\begin{remark}
    Actually, as for the additive sewing lemma (which is itself a subcase of the
    multiplicative sewing lemma), we have a stronger statement: No 
    (non-linear) flow satisfies $d(\psi_{t,s}(a),a\alpha_{s,t})\leq C\varpi(\omega_{s,t})$
    except $\Set{a\mapsto a\beta_{s,t}}_{(s,t)\in\TT_+^2}$ which is constructed
    as the limit of the products of the $\alpha_{s,t}$ over smaller and smaller intervals. 
\end{remark}

\subsection{The multiplicative sewing lemma in a Banach algebra}

Consider now that $\uV$ has a Banach algebra structure  with a norm $\eucnorm{\cdot}$
such that $\eucnorm{ab}\leq \eucnorm{a}\times\eucnorm{b}$ and a unit element $1$
(the product of two elements is still denoted by $ab$).

A typical example is the Banach algebra of bounded operators over a Banach space~$\uX$.

This situation fits in the multiplicative sewing lemma with $d(a,b)=\eucnorm{a-b}$
and $N(a)=\eucnorm{a}$, $a,b\in\uV$. As seen in \cite{coutin-lejay1}, we have many 
more properties: continuity, existence of an inverse, Dyson formula, Duhamel principle, ...

In particular, this framework is well suited for considering linear differential equations of type
\begin{equation*}
y_{t,s}=a+\int_s^t y_{r,s}\vd \alpha_r, a\in\uV,\ t\geq s\geq 0
\end{equation*}
for an operator valued path $\alpha:\rTT_2\to\uV$. If $\alpha$ is $\gamma$-Hölder with $\gamma>1/2$, 
then $\phi_{t,s}(a)=a(1+\alpha_{s,t})$ defines an almost flow which satisfies the condition UL 
(at the price of imposing some conditions on $\alpha$, this could be extended to $\gamma<1/2$).

Defining an ``affine flow'' $\phi_{t,s}(a)=a(1+\alpha_{s,t})+\beta_{s,t}$
where both $\alpha$ and $\beta$ are $\gamma$-Hölder with $\gamma>1/2$, the associated flow 
$\psi$ is such that $\psi_{t,s}(a)$ is solution to the perturbed equation
\begin{equation*}
    \psi_{t,s}(a)=a+\int_s^t \psi_{r,s}(a)\vd \alpha_r+\beta_{s,t}.
\end{equation*}
This gives an alternative construction to the one of \cite{coutin-lejay1} where
a \emph{backward integral} between $\beta$ and $\alpha$ was defined in the
style of the Duhamel formula. All these results are extended to the rough case $1/3<\gamma\leq 1/2$.

\begin{example}[Lyons extension theorem]
With the tensor product $\otimes$ as product and a suitable norm, 
for any integer $k$, the tensor algebra
\begin{equation*}
\uT_k(\uX)\eqdef\RR\oplus\uX\oplus\uX^{\otimes 2}\oplus\dotsb\oplus \uX^{\otimes k}
\end{equation*}
is a Banach algebra.
Chen series of iterated integrals (and then rough paths) take their values
in some space $\uT_k(\uX)$. The Lyons extension theorem states that 
any rough path $\bx$ of finite $p$-variation with values in $\uT_k(\uX)$ for some $k\geq \lfloor p \rfloor$
is uniquely extended to a rough path with values in $\uT_\ell(\uX)$ for any $\ell\geq k$,
which leads to the concept of \emph{signature}
\cite{lyons98a,lyons02b}.
This follows  $a\mapsto a\otimes \bx_{s,t}$
as an almost flow which satisfies the UL condition (see also~\cite{feyel}
and also~\cite{coutin-lejay1}).
\end{example}


\subsection{Rough differential equation}

Now, we show that our construction is related to the one of A.M. Davie
\cite{davie05a}. The main idea of Davie was to construct solutions as 
paths $y:\TT\to\uV$ that satisfies~\eqref{eq:def_davie} for a suitable
\textquote{algorithm} $\phi_{t,s}$ of the solution 
between $s$ and $t$. Solutions passing through $a$ in $0$ 
are then constructed as limit of 
using $\Set{\phi^\pi_{t,0}(a)}_{\pi}$ (See Proposition~\ref{prop:sel:1}).
The algorithm $\phi_{t,s}$ is given by a truncated
Taylor expansion of the solution of \eqref{eq:rde:3}.  The number of terms 
to consider in the Taylor expansion depends directly on the regularity of~$x$. In the Young
case one term is needed whereas in the rough case two terms are required.

In this section, we show that the algorithms provided in \cite{davie05a}
are almost flows under the same regularity on the vector field $f$
and the path $x$. Not only we recover existence of D-solutions, 
but we also show that measurable flows exist when 
the vector fields $f$ are~$\cCb^1$ (Young case) or~$\cCb^2$ (rough case)
in situation in which non-uniqueness of solutions is known to hold, 
again due to \cite{davie05a}, unless $f$ is of class $\cCb^{1+\gamma}$
(Young case) or $\cCb^{2+\gamma}$ (rough case).

Here $\uU$ and $\uV$ are two Banach spaces, where we use the same notation
$\abs{\cdot}$ for their norms. We denote by
$\cL(\uU,\uV)$ the continuous linear maps from $\uU$ to $\uV$. Let $f$ be a map
from $\uV$ to $\cL(\uU,\uV)$. If $f$ is regular, we denote its Fréchet derivative in
$a\in\uV$, $\vd f(a)\in \cL(V,\cL(U,V))$.

Moreover, for any $a\in \uW$ and $(r,s,t)\in\rTT_3$, we set $\phi_{t,s,r}(a):=\phi_{t,s}\circ\phi_{s,r}(a)-\phi_{t,r}(a)$.

\subsubsection{Almost flow in the Young case}

Let $x:\TT\rightarrow \uU$ be a path of finite $p$-variation controlled by $\omega$ with $1\leq p <2$.

We define a family $(\phi_{t,s})_{(s,t)\in\rTT_2}$ in $\cF(\uV)$ such that for all $a\in\uV$ and $(s,t)\in\rTT_2$,
\begin{equation}
\phi_{t,s}(a):= a+f(a)x_{s,t},
\end{equation}
where $x_{s,t}:=x_t-x_s$. 

\begin{proposition}
\label{prop:almost_flow_young_case}
Assume that $f\in\cC^{\gamma}(\uV,\cL(U,V))$, with $1+\gamma>p$. Then $\phi$ is an almost flow.
\end{proposition}
\begin{proof}
We check that assumptions of Definition~\ref{def:almost_flow} hold. Let $(r,s,t)$ be in $\rTT_3$ and let $a,b$ be in $\uV$.
First, $\phi_{t,t}(a)=a$ because $x_{t,t}=0$.
Second, 
\begin{equation*}
    \abs{\phi_{t,s}(a)-a}
\leq \abs{f(a)}\cdot\abs{x_{s,t}}
\leq\abs{f(a)}\cdot \normp{x}\omega_{s,t}^{1/p},
\end{equation*}
which proves \eqref{eq:h1}. Third, 
\begin{equation*}
\abs{\phi_{t,s}(a)-\phi_{t,s}(b)}\leq \abs{a-b}+\abs{f(a)-f(b)}|x_{s,t}|
\leq \abs{a-b}+\normf{\gamma}{f}\normp{x}\omega_{s,t}^{1/p}\abs{a-b}^\gamma,
\end{equation*} 
which proves \eqref{eq:h2}. It remains to prove \eqref{eq:h3}.
Since 
\begin{align*}
\phi_{t,s,r}(a)=f(\phi_{s,r}(a))x_{s,t}-f(a)x_{s,t},
\end{align*}
we obtain 
\begin{multline*}
|\phi_{t,s,r}(a)|\leq \normf{\gamma}{f} \normp{x}\omega_{s,t}^{1/p}|\phi_{s,r}(a)-a|^\gamma
\leq \normf{\gamma}{f}\abs{f(a)}^\gamma\normp{x}^2 \omega_{r,t}^{(1+\gamma)/p}\\
\leq \normf{\gamma}{f}(1+\abs{f(a)})\normp{x}^2\omega_{r,t}^{(1+\gamma)/p}.
\end{multline*}
Setting $\varpi(\omega_{r,t}):=\omega_{r,t}^{(1+\gamma)/p}$, 
$\eta(\omega_{s,t}):=\normf{\gamma}{f}\normp{x}\omega_{s,t}^{1/p}$
and 
\begin{equation*}
N_\gamma(a):=\left(1+\abs{f(a)}\right)\left(\normp{x}+\normf{\gamma}{f}\normp{x}^2\right),
\end{equation*}
it proves that $\phi$ is an almost flow.

This concludes the proof.
\end{proof}

Let $\psi$ be a flow in the same galaxy as the almost flow $\phi$.
For any $a\in\uV$ and any $(r,t)\in\TT_+^2$, we set 
\begin{equation*}
    y_t(r,a):=\psi_{t,r}(a)\text{ so that }y_r(r,a)=a.
\end{equation*}
Clearly, $(r,a)\mapsto (t\in[r,T]\mapsto y_t(r,a))$ is a family of continuous paths 
which satisfies
\begin{equation*}
    \abs{y_t(r,a)-\phi_{t,r}(y_s(r,a))}\leq CN_\gamma(y_s(r,a))\omega_{s,t}^{2/p},\ \forall (s,t)\in\TT_+^2,\ \forall a\in\uV
\end{equation*}
since $y_t(r,a)=\psi_{t,s}(y_s(r,a))$. Besides, $s\in[r,T]\mapsto N_\gamma(y_s(r,a))$
is bounded. Therefore, with our choice of the almost flow $\phi$, 
$\psi_{\cdot,r}(a)=y(r,a)$ is a solution in the sense defined
by A.M.~Davie~\cite{davie05a} for the Young differential equation 
$z_t=a+\int_r^t f(z_s)\vd x_s$. Even if several solutions may exist 
for a given $(r,a)$, the flow corresponds to a particular choice
of a family of solutions which is constructed thanks to a selection 
principle. This family of solution is stable under splicing
(see Definition~\ref{def:splicing}).

\begin{corollary}
    We assume that $\uV$ is a finite-dimensional vector space and
    $f\in\cCb^{1}(\uV,\cL(\uU,\uV))$.
Then there exists 
    a flow $\psi\in\cF(\uV)$ in the same galaxy as $\phi$ such that 
    $\psi_{t,s}$ is Borel measurable
    for any $(s,t)\in\TT_+^2$.
\end{corollary}
\begin{remark}
    When $f\in\cCb^\gamma$, 
    several D-solutions to the Young differential equation
    $y=a+\int_0^\cdot f(y_s(a))\vd x_s$ may exist (Example~1 in~\cite{davie05a}). 
    Uniqueness arises when $f\in\cCb^{1+\gamma}$ with $1+\gamma>p$. 
    Hence, a measurable flow may exist even when several D-solution may exist.
\end{remark}
\begin{proof}
According to Proposition~\ref{prop:almost_flow_young_case}, $\phi$ is an almost flow. Here $\gamma=1$, so $\phi$ is Lipschitz. Then, we conclude the proof in applying Theorem~\ref{thm:nls:1} to $\phi$.
\end{proof}

\subsubsection{Almost flow in the rough case}

When the regularity of $x$ is weaker than in the Young case,
we need more terms in the Taylor expansion to obtain an almost flow.

Let $T_2(\uU)\eqdef\RR\oplus\uU\oplus(\uU\otimes\uU)$
be the truncated tensor algebra (with addition $+$ and tensor product $\otimes$).
A distance is defined on the subset of elements of $T_2(\uU)$ 
on the form $a=1+a^1+a^2$ with $a^i\in\uU^{\otimes i}$ by 
$d(a,b)=\abs{a^{-1}\otimes b}$ where $\abs{\cdot}$
is a norm on $T_2(\uU)$ such that $\abs{a\otimes b}\leq \abs{a}\cdot\abs{b}$
for any $a,b\in\uU$.

Let $\bx=(1,\bx^1,\bx^2)$ be a rough path with values in $T_2(\uU)$ of finite $p$-variation, 
$2\leq p<3$, controlled by $\omega$ (see \textit{e.g.}, \cite{lyons02b,friz14a}
for a complete definition).

We define a family $(\phi_{t,s})_{(s,t)\in\rTT_2}$ in $\cF(\uV)$ such that for all $a\in\uV$ and $(s,t)\in\rTT_2$,
\begin{equation}
    \label{eq:rde:4}
\phi_{t,s}(a):=a+f(a)\bx_{s,t}^1+\vd f(a)\cdot f(a)\bx^{2}_{s,t}.
\end{equation}

\begin{proposition}
\label{prop:davie_rough}
Assume that $f\in\cCb^{1+\gamma}(\uV,\cL(\uU,\uV))$, with $2+\gamma>p$. 
Then $\phi$ is an almost flow.
\end{proposition}
\begin{proof}
We check that the assumptions of Definition~\ref{def:almost_flow} hold.
The proofs of \eqref{eq:h0}, \eqref{eq:h1} and \eqref{eq:h2} are very similar 
to the ones in the proof of Proposition~\ref{prop:almost_flow_young_case}. The computation
to show \eqref{eq:h3} is a bit more involved.

Indeed, for any $a\in\uV$, $(r,s,t)\in\rTT_3$,
\begin{align*}
\phi_{t,s,r}(a)=&-f(a)\bx_{s,t}^1+f(\phi_{s,r}(a))\bx_{s,t}^1-\vd f(a)\cdot f(a)(\bx_{s,t}^2+\bx_{r,s}^1\otimes \bx_{s,t}^1)\\
&+\vd f(\phi_{s,r})\cdot f(\phi_{s,r}(a))\bx_{s,t}^2\\
=&[f(\phi_{s,t}(a))-f(a)-\vd f(a)\cdot f(a)\bx^1_{r,s}]\otimes \bx_{s,t}^1\\
&+[\vd f(\phi_{s,r}(a))\cdot f(\phi_{s,r}(a))-\vd f(a)\cdot f(a)]\bx^2_{s,t}\\
=&\underbrace{f(\phi_{s,t}(a))-f(a)-\vd f(a)\cdot (\phi_{s,r}(a)-a)}_{\run_{r,s,t}}+\underbrace{\vd f(a)\cdot f(a)\bx^2_{r,s}\otimes \bx^1_{s,t}}_{\rdeux_{r,s,t}}\\
&+\underbrace{[\vd f(\phi_{s,t}(a))\cdot f(\phi_{s,r}(a))-\vd f(a)\cdot f(a)]\bx^2_{s,t}}_{\rtrois_{r,s,t}}.
\end{align*}

For the first term,
\begin{align*}
|\run_{r,s,t}|&\leq  \normf{\gamma}{\vd f}\normp{\bx^1}\omega_{s,t}^{1/p}|\phi_{s,r}(a)-a|^{1+\gamma}\\
&\leq \normf{\gamma}{\vd f}\normp{\bx^1}\omega_{s,t}^{1/p} [\normsup{f}\normp{\bx^1}+
\normsup{\vd f\cdot f}\normpp{\bx^2}\omega_{0,T}^{1/p}]^{1+\gamma}
\omega_{r,s}^{(1+\gamma)/p}.
\end{align*}

For the two last terms,
\begin{equation*}
|\rdeux_{r,s,t}|\leq
\normsup{\vd f\cdot f}\normp{\bx^1}\normpp{\bx^2} \omega_{r,t}^{3/p}
\leq \normsup{\vd f\cdot f}\normp{\bx^1}\normpp{\bx^2}\omega_{0,T}^{(1-\gamma)/p} \omega_{r,t}^{(2+\gamma)/p}
\end{equation*}
and
\begin{equation*}
|\rtrois_{r,s,t}|\leq \normpp{\bx^2}\omega_{r,t}^{2/p}
    \left[\normf{\gamma}{\vd f}\normsup{f}\abs{\phi_{s,r}(a)-a}^\gamma+\normsup{\vd f}\normlip{f}
\abs{\phi_{s,r}(a)-a}\right]
\leq C\omega_{r,t}^{3/p}, 
\end{equation*}
where $C$ is a constant which depends on $f$, $\vd f$, $\omega$, $\gamma$, $\bx$. It proves that $\phi$ is a Lipschitz almost flow.

This concludes the proof.
\end{proof}

As for the Young case, any flow $\psi$ in the same galaxy
as the almost flow $\phi$ given by \eqref{eq:rde:4} gives
rise to a family of solutions to the RDE $z_t=a+\int_0^t f(z_s)\vd\bx_s$.

\begin{corollary}
    \label{cor:10}
    We assume that $\uV$ is a finite-dimensional vector space and
    $f\in\cCb^{2}(\uV,\cL(\uU,\uV))$.
    Then there exists 
    a flow $\psi\in\cF(\uV)$ in the same galaxy as $\phi$ such that $\psi_{t,s}$ is 
    Borel measurable for any $(s,t)\in\TT_+^2$.
\end{corollary}
\begin{remark} 
    When $f\in\cCb^{1+\gamma}$, 
    several D-solutions to the RDE $y=a+\int_0^\cdot f(y_s(a))\vd \bx_s$ may exist (Example~2 in~\cite{davie05a}).
    Uniqueness requires $f$ to be $(2+\gamma)$-Hölder continuous with $2+\gamma>p$. 
    Hence, Corollary~\ref{cor:10} shows that a measurable flow exists even 
    when several D-solutions may exist.
\end{remark}
\begin{proof}
According to Proposition~\ref{prop:davie_rough}, $\phi$ is an almost flow. Here $\gamma=1$, so $\phi$ is Lipschitz. 
Then, we conclude the proof in applying Theorem~\ref{thm:nls:1} to $\phi$.
\end{proof}

\section*{Acknowledgments}

The authors are very grateful to Laure Coutin for her numerous valuable remarks and corrections. 
We are also grateful to the CIRM (Marseille, France) for its kind hospitality with the Research-in-Pair program.
Finally, we thank the referees for their careful reading.

\printbibliography

\end{document}